\documentclass[twoside, final]{amsart} 

\usepackage{amsmath, amsfonts, amsthm, amssymb}
\usepackage{exscale}
\usepackage{cite}
\usepackage{epsfig}
\usepackage{amscd}
\usepackage{graphics}
\usepackage{color}
\usepackage{dsfont}

\usepackage[notref, notcite]{showkeys}

\renewcommand{\geq}{\geqslant}
\renewcommand{\leq}{\leqslant}

\newtheorem{theorem}{Theorem}

\newtheorem{proposition}{Proposition}[section]

\newtheorem{corollary}[proposition]{Corollary}
\newtheorem{lemma}[proposition]{Lemma}
\newtheorem*{main-theorem}{Main Theorem}
\newtheorem*{theorem*}{Theorem}

\theoremstyle{definition}

\newtheorem{remark}[proposition]{Remark}

\newtheorem*{remark*}{Remark}

\numberwithin{equation}{section}

\def\phi{\varphi}

\def\reals{{\mathbb R}}
\def\cx{{\mathbb C}}

\def\Ci{{\mathcal C}^\infty}
\def\Re{\,\mathrm{Re}\,}
\def\Im{\,\mathrm{Im}\,}

\def\supp{\mathrm{supp}\,}

\def\id{\,\mathrm{id}\,}

\def\O{{\mathcal O}}
\def\SS{{\mathbb S}}
\def\s{{\mathcal S}}

\def\phi{\varphi}

\def\be{\begin{eqnarray*}}
\def\ee{\end{eqnarray*}}
\def\ben{\begin{eqnarray}}
\def\een{\end{eqnarray}}

\def\lll{\left\langle}
\def\rrr{\right\rangle}

\def\EE{\mathcal{E}}
\def\L2R{L_{\text{Rest}}^2}

\def\11{\mathds{1}}

\def\HH{\mathcal{H}}

\def\tchi{\tilde{\chi}}
\def\L2c{L^2_{\text{comp}}}

\def\tDelta{\widetilde{\Delta}}
\def\tV{\widetilde{V}}

\def\tP{\widetilde{P}}
\def\tp{\tilde{p}}
\def\tu{\tilde{u}}

\def\tA{\tilde{A}}

\def\tR{\tilde{R}}

\def\Vol{\text{Vol}}

\def\B{\mathcal{B}}
\def\tg{\tilde{g}}
\def\Slo{S_{\text{lo}}}

\def\Slojp{S_{\text{lo},j,+}}
\def\Slojm{S_{\text{lo},j,-}}
\def\Slojpm{S_{\text{lo},j,\pm}}
\def\Sloo{S_{\text{lo},0,0}}

\def\Shi{S_{\text{hi}}}
\def\tE{\widetilde{E}}
\def\p{\partial}

\begin{document}

\title[Strichartz estimates with Loss]{Near Sharp Strichartz estimates with loss
  in the presence of degenerate hyperbolic trapping}

\author[Christianson]{Hans~Christianson}
\email{hans@math.unc.edu}
\address{Department of Mathematics, UNC-Chapel Hill \\ CB\#3250
  Phillips Hall \\ Chapel Hill, NC 27599}

\subjclass[2000]{}

\begin{abstract}

We consider an $n$-dimensional spherically symmetric, asymptotically Euclidean manifold with
two ends and a codimension $1$ trapped set which is degenerately
hyperbolic.  By separating variables and constructing a semiclassical parametrix for
a time scale polynomially beyond Ehrenfest time, we show that solutions to the linear Schr\"odiner
equation with initial conditions localized on a spherical harmonic satisfy Strichartz estimates with a loss depending only on the
dimension $n$ and independent of the degeneracy.  The Strichartz estimates are
sharp up to an arbitrary $\beta>0$ loss.   This is in contrast
to \cite{ChWu-lsm}, where it is shown that solutions satisfy a sharp local
smoothing estimate with loss depending only on the degeneracy of the
trapped set, independent of the dimension.

\end{abstract}

\maketitle

\section{Introduction}
\label{S:intro}

It is well known that there is an intricate interplay between the
existence of {\it trapped geodesics}, those which do not escape to
spatial infinity, and {\it dispersive estimates} for the associated
quantum evolution.  Trapping can occur in many different ways, from a
single trapped geodesic (see \cite{Bur-sm, BuZw-bb, Chr-NC, Chr-QMNC,
  Chr-disp-1}), to a thin fractal trapped set (see \cite{NoZw-qdr,
  Chr-wave-2, Chr-sch2, Dat-sm}), to codimension $1$ trapped sets in
general relativity (see, for example,
\cite{BlSo-decay,BoHa-decay,DaRo-RS,MMTT-Sch,Luk-decay,TaTo-Kerr,LaMe-decay,WuZw-nhre}
and the references therein), to elliptic trapped
sets and boundary value problems.  Dispersive type estimates also come
in many flavors, but are all designed to express in some manner that
the mass of a wave function tends to spread out as the wave function
evolves.  Since the mass of wave functions tends to move along the
geodesic flow, the presence of trapped geodesics suggests some
residual mass may not spread out, or may spread out more slowly than
in the non-trapping case.  In this paper, we concentrate on {\it
  Strichartz estimates}, and exhibit a class of manifolds for which we
prove near-sharp Strichartz estimates with a loss depending only on
the dimension of the trapped set.  This class of manifolds has already
been studied in \cite{ChWu-lsm}, where a sharp local smoothing
estimate is obtained with a loss depending only on how flat the
manifold is near the trapped set.  This presents an interesting
dichotomy conjecture: ``loss in local smoothing depends only on the
kind of trapping, while loss in Strichartz estimates depends only on
the dimension of trapping''.

The purpose of this paper is to
study a very simple class of manifolds with a hypersphere of trapped
geodesics.  If the dynamics near such a sphere are strictly
hyperbolic in the normal direction, then resolvent estimates are already obtained in
\cite{WuZw-nhre} (see also \cite{Chr-disp-1,Chr-sch2}) which can be used to prove local smoothing estimates
with only a logarithmic loss.
However, if the dynamics are only weakly 
hyperbolic, resolvent estimates and local smoothing estimates are
obtained in \cite{ChWu-lsm} with a sharp polynomial loss in both.  We
now turn our attention to studying Strichartz estimates, which are
mixed $L^p L^q$ time-space estimates.  The typical procedure for
proving Strichartz estimates is to construct parametrices (approximate
solutions), which encode how wave packets move with the geodesic
flow.  For solutions of the Schr\"odinger equation, wave packets at
higher frequency move at a higher velocity, so the presence of trapping, or more precisely of conjugate
points means that parametrices can typically be constructed only for
time intervals {\it depending on the frequency} of the wave packet.
Then summing up many parametrices to get an estimate on a fixed time scale
leads to derivative loss in Strichartz estimates.

However, if the trapped set is sufficiently thin and hyperbolic, we
expect that {\it most} of a wave packet still propagates away quickly,
and a
procedure developed by Anantharaman \cite{Anan-entropy} allows one to
exploit this to logarithmically extend the timescale on which one can construct a
parametrix leading to Strichartz estimates with no loss 
\cite{BGH}.

For the manifolds studied in this paper, the trapping is degenerately
hyperbolic, so we still expect some mass of each wave packet to
propagate away, but at a much slower rate than the strictly hyperbolic
case.  As a consequence, we need to extend the parametrix {\it
  polynomially} in time to get sharp estimates.  The techniques in
\cite{Anan-entropy,BGH} will not work in this situation since the
$\O(h^\infty)$ estimate of decaying correlations will not control the
{\it exponential} number of such correlations.  
In this paper, we fail to prove estimates all the way to the sharp
polynomial timescale, but we are nevertheless able to extend the parametrix
construction to the sharp timescale up to an arbitrary $\beta>0$ loss, which is
expressed as a loss in derivative in the main theorem.
Further, the technique of proof involves decomposing the
solution in terms of spherical harmonics in order to reduce the
problem to a $1$-dimensional semiclassical parametrix construction.  Lacking a
square-function estimate for spherical harmonics, the proof only works
for initial data localized along one spherical harmonic eigenspace.  In this
sense, the result shows more about the {\it natural semiclassical
timescale}, polynomially extended beyond Ehrenfest time, for which we have good control of
the rate of dispersion.

We begin by describing the geometry.  We consider $X = \reals_x \times \SS^{n-1}_\theta$, equipped with the metric
\[
g = d x^2 + A^2(x) G_\theta,
\]
where $A \in \Ci$ is a smooth function, $A \geq \epsilon>0$ for some
epsilon, and $G_\theta$ is the metric on $\SS^{n-1}$.    
From this metric, we get the volume form
\[
d \Vol = A(x)^{n-1} dx d \sigma,
\]
where $\sigma$ is the usual surface measure on the sphere.  The
Laplace-Beltrami operator acting on $0$-forms is computed: 
\[
\Delta f = (\partial_x^2 + A^{-2} \Delta_{\SS^{n-1}} + (n-1) A^{-1}
A' \partial_x) f,
\]
where $\Delta_{\SS^{n-1}}$ is the (non-positive) Laplace-Beltrami
operator on the sphere.

We study the case $A(x) = (1 + x^{2m})^{1/2m}$, $m \geq 2$, in which case the manifold is asymptotically Euclidean
(with two ends), and has a trapped hypersphere at the unique critical point $x =
0$.  Since $A(x)$ has a degenerate minimum at $x = 0$, the trapped
sphere is {\it weakly} normally hyperbolic in
the sense that it is unstable and isolated, but degenerate (see Figure
\ref{fig:fig1}).

Our main theorem is the following, which expresses that a solution of
the linear homogeneous Schr\"odinger equation on this manifold
satisfies Strichartz estimates with loss depending only on the
dimension $n$, up to an arbitrary $\beta>0$ loss.

\begin{theorem}
\label{T:T1a}
Suppose $u$ solves
\begin{equation}
\label{E:Sch-1}
\begin{cases}
(D_t + \Delta ) u = 0, \\
u|_{t=0} = u_0,
\end{cases}
\end{equation}
where $u_0 = H_k u_0$ is localized on the spherical harmonic subspace of
order $k$.  
Then for any $T, \beta >0$, there exists $C_{T,\beta}>0$ such that 
\begin{equation}
\label{E:main-str}
\| u \|_{L^p([0,T]) L^{q}(M)} \leq C_{T,\beta} \| \lll D_\theta
\rrr^{(n-2)/pn + \beta} u_0 \|_{L^2(M)},
\end{equation}
where
\[
\frac{2}{p} + \frac{n}{q} = \frac{n}{2},
\]
and $2 \leq q < \infty$ if $n = 2$.

\end{theorem}

\begin{remark}
There are several important observations to make about Theorem
\ref{T:T1a}.  First, this theorem concerns {\it endpoint} Strichartz
estimates. 
In dimension $n \geq 3$, if we take $\beta < 1/n$, the
loss in derivative is then $(n-2)/2n  + \beta < 1/2$; that is, the loss
is always less than the loss following the argument of
Burq-G\'erard-Tzvetkov \cite{BGT-comp} (which gives a loss of $1/2$
for endpoint estimates in $n \geq 3$).  Second, there is only a
$\beta>0$ loss over the Euclidean (scale-invariant) estimates 
if $n = 2$, that is, if the trapped set is a single degenerate
periodic geodesic, so we can get as close to the no-loss estimates as
we like.  We expect the $\beta>0$ derivative loss can actually be removed in
all dimensions, but this is beyond our techniques.  Third, in all dimensions, the loss depends {\it only on the
  dimension of the trapped set}.  It does not depend on $m$, the order
of degeneracy of the trapping.  This is in sharp contrast to the local
smoothing effect, which depends only on $m$, and {\it not} on the
dimension $n$ (see \cite{ChWu-lsm} and below).


For dimensions $n \geq 3$, the estimate \eqref{E:main-str}
is near sharp on natural semiclassical time scales (see Corollary
\ref{C:C1a}), in the sense that no better {\it polynomial} derivative
estimate can hold.  In dimension $n = 2$, the same is true by
comparing to the scale-invariant case.

Finally, since $u_0$ is localized to a single spherical harmonic, the
estimate in the theorem can be rephrased, since
\[
\| \lll D_\theta \rrr^{(n-2)/pn + \beta} u_0
\|_{L^2(M)} \sim \| \lll k \rrr^{(n-2)/pn + \beta} u_0 \|_{L^2(M)}.
\]
\end{remark}

\begin{figure}
\hfill
\centerline{\input{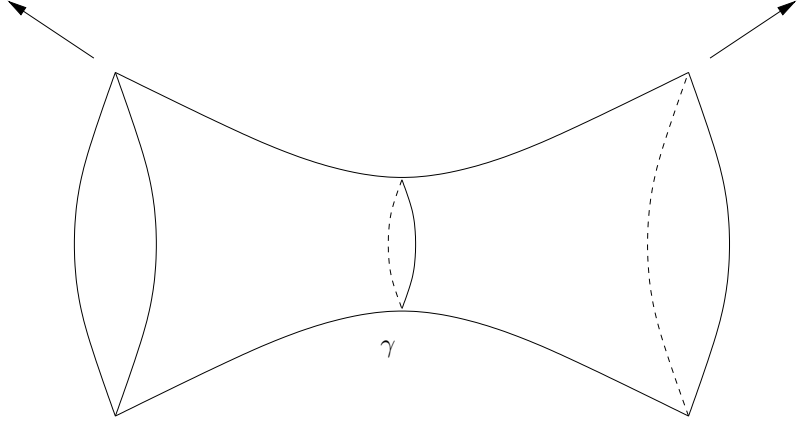}}
\caption{\label{fig:fig1} A piece of the manifold $X$ and the trapped
  sphere at $x=0$.   }
\hfill
\end{figure}

\subsection*{Acknowledgements}
The author would like to thank Nicolas Burq, Kiril Datchev, Colin Guillarmou, Jeremy
Marzuola, Jason Metcalfe, Fabrice Planchon, and Michael Taylor for
helpful and stimulating conversations during the preparation of this manuscript.

\section{Reduction in dimension}

In this section we use a series of known techniques and estimates to
reduce study of the Schr\"odiner equation on $M$ to the study of a
semiclassical Schr\"odiner equation on $\reals$ with potential.  The
potential has a degenerate critical point, and we use a technical
blow-up calculus to construct a sequence of parametrices near the
critical point.

We observe that we can conjugate
$\Delta$ by an isometry of metric spaces and separate variables so
that spectral analysis of $\Delta$ is equivalent to a one-variable
semiclassical problem with potential.  That is, let $T : L^2(X, d
\Vol) \to L^2(X, dx d \theta)$ be the isometry given by
\[
Tu(x, \theta) = A^{(n-1)/2}(x) u(x, \theta).
\]
Then $\tDelta = T \Delta T^{-1}$ is essentially self-adjoint on $L^2 (
X, dx d \sigma)$ for our choice of $A$.  A simple calculation gives
\[
-\tDelta f = (- \partial_x^2 - A^{-2}(x) \Delta_{\SS^{n-1}} + V_1(x) )
f,
\]
where the potential
\[
V_1(x) = \frac{n-1}{2} A'' A^{-1} + \frac{(n-1)(n-3)}{4} (A')^2 A^{-2}.
\]

Of course, conjugating the Laplacian by an $L^2$ isometry does not
necessarily preserve $H^s$ or $L^q$ spaces.

\begin{lemma}
\label{L:T-conj-lemma}
With the notation $A(x) = (1 + x^{2m})^{1/2m}$ from above, for $s \geq 0$,
\[
\| Tu \|_{H^s(dx d \sigma)} \leq C \| u \|_{H^s(d\Vol)},
\]
\[
\| \lll -\Delta_{\SS^{n-1}} \rrr^s Tu \|_{L^2(dx d \sigma )} = \| \lll -\Delta_{\SS^{n-1}} \rrr^s u
\|_{L^2(d\Vol)},
\]
and for $q \geq 2$,
\[
\| u \|_{L^q(d\Vol)} \leq C \| Tu \|_{L^q(dx d \sigma)}.
\]
\end{lemma}
\begin{proof}
The result for $0 < s \leq 1$ follows from the $L^2$ and $H^1$ result,
which follows by observing that
\[
\partial_x A^{(n-1)/2} u = A^{(n-1)/2} \partial_x u + \frac{(n-1)}{2}
A' A^{(n-3)/2} u.
\]
But since $|A'| \leq C A$, the $L^2(dx d \sigma)$ norm of $\partial_x Tu$ is
bounded by the $H^1(d\Vol)$ norm of $u$.

The result for angular derivatives follows by commuting with
$A^{(n-1)/2}(x)$.  

For the $L^q$ result, we compute
\[
\int | u(x, \theta) |^q A^{(n-1)}(x) dx d \sigma = \int | A^{(n-1)(1/q
  - 1/2)} T u(x, \theta) |^q dx d \theta.
\]
The function $A^{(n-1)(1/q
  - 1/2)}(x)$ is bounded for $q \geq 2$, so the $L^q$ inequality is
  true as well.
\end{proof}

As a consequence of this, to prove Theorem \ref{T:T1a}, it suffices to
prove the following Proposition, and apply Lemma \ref{L:T-conj-lemma}
with $v = Tu$.

\begin{proposition}
\label{P:P1a}
Suppose $v$ solves
\begin{equation}
\label{E:Sch-1a}
\begin{cases}
(D_t + \tDelta ) v = 0, \\
v|_{t=0} = v_0,
\end{cases}
\end{equation}
where $v_0 = H_k v_0$ is localized on the spherical harmonic subspace of
order $k$.  Then for any $T >0$ and $\beta>0$, there exists $C_{T,\beta}>0$ such that 
\begin{equation}
\label{E:main-str1a}
\| v \|_{L^p([0,T]) L^{q}} \leq C_{T,\beta} \| \lll D_\theta \rrr^{(n-2)/pn + \beta} v_0 \|_{L^2},
\end{equation}
where
\[
\frac{2}{p} + \frac{n}{q} = \frac{n}{2},
\]
and $2 \leq q < \infty$ if $n = 2$.

\end{proposition}

We now separate variables by projecting onto the $k$th spherical
harmonic eigenspace.  That is, let $\HH_k$ be the $k$th eigenspace of
spherical harmonics, so that $v \in \HH_k$ implies
\[
-\Delta_{\SS^{n-1}} v = \lambda_k^2 v,
\]
where
\[
\lambda_k^2 = k(k+n-2).
\]
Let $H_k : L^2(\SS^{n-1}) \to \HH_k$ be the projector.

Since $v_0$ is assumed to satisfy $v_0 = H_k v_0$ for some $k$ and the
conjugated 
Laplacian preserves spherical harmonic eigenspaces, we have also $v =
H_k v$.  Motivated by spectral theory, we
compute:

\[
(-\tDelta- \lambda^2) v =  P_k v,
\]
where
\begin{equation}
\label{E:Pk}
P_k v = P_k H_k v = (-\frac{\partial^2}{\partial x^2} + k(k+n-2)
A^{-2}(x) + V_1(x) - \lambda^2) v.
\end{equation}
Setting $h = (k(k+n-2))^{-1/2}$ and rescaling, we have the one-dimensional semiclassical operator
\[
P(z,h) \psi(x) = (-h^2 \frac{d^2}{dx^2} + V(x) -z) \psi(x),
\]
where the potential is
\[
V(x) = A^{-2}(x) + h^2 V_1(x)
\]
and the spectral parameter is $z = h^2 \lambda^2$.

For our case where $A(x) = (1 + x^{2m})^{1/2m}$, the subpotential $h^2
V_1$ is seen to be lower order in both the semiclassical and
scattering sense.  Furthermore, the principal potential $A^{-2}(x)$
is even, smooth, decays like $x^{-2}$ at $\pm \infty$ and has a unique 
degenerate maximum of the form $1 - x^{2m}$ at $x = 0$.



\section{Endpoint Strichartz estimates}
Before proceeding to the endpoint Strichartz estimates, let us briefly
recall the local smoothing estimates which will eventually allow us to
glue together Strichartz estimates on semiclassical timescales.

\subsection{Local smoothing estimates}

In this subsection, we recall the local smoothing estimates from
\cite{ChWu-lsm}, as well as the dual versions which
we will use in this paper.
\begin{theorem}[\cite{ChWu-lsm}]
\label{T:lsm}
Let $V(x) = A^{-2}(x) + h^2 V_1(x)$ as above.  Then for any $T >0$,
there exists a constant $C = C_T >0$ such that
\[
\int_0^T \| | x |^{m-1} \lll x \rrr^{-m-1-3/2} e^{i t (-\p_x^2 +
  h^{-2} V )} u_0 \|_{L^2}^2 dt \leq C h \| u_0 \|_{L^2}^2,
\]
and
\[
\int_0^T \| \lll x \rrr^{-3/2} e^{i t (-\p_x^2 +
  h^{-2} V )} u_0 \|_{L^2}^2 dt \leq C h^{1/(m+1)} \| u_0 \|_{L^2}^2.
\]
\end{theorem}
The dual versions of these estimates are given in the following Corollary.
\begin{corollary}
Let $V(x) = A^{-2}(x) + h^2 V_1(x)$ as above.  Then for any $T >0$,
there exists a constant $C = C_T >0$ such that
\[
\left\|\int_0^T | x |^{m-1} \lll x \rrr^{-m-1-3/2} e^{-i t (-\p_x^2 +
  h^{-2} V )} f dt \right\|_{L^2}^2  \leq C h \| f \|_{L^2_T L^2}^2,
\]
and
\[
\left\|\int_0^T  \lll x \rrr^{-3/2} e^{-i t (-\p_x^2 +
  h^{-2} V )} f dt \right\|_{L^2}^2 \leq C h^{(1-m)/(m+1)} \| f \|_{L^2}^2.
\]
\end{corollary}

The purpose of these results is to demonstrate that there is perfect
$1/2$ derivative local smoothing away from $x = 0$, or local smoothing
with either a loss in derivative or with a vanishing multiplier at $x
= 0$.

\subsection{The endpoint Strichartz estimates}

The endpoint Strichartz estimates are the $L^2_T
L^{2^\star}$\footnote{Throughout this manuscript, we use the notation
  $L_T^p L^q = L^p([0,T]) L^q( M)$ to denote the local in time, global
in space Strichartz norm.}
estimates, where $2^\star$ is the Strichartz dual:
\[
1 + \frac{n}{2^\star} = \frac{n}{2} ,
\]
which implies $2^\star = 2n/(n-2)$ for $n \geq 3$, and $2^\star =
\infty$ if $n = 2$.   



We want to estimate $v$ in
$L^{2^\star}(M)$, which we do using the following estimate due to
Sogge \cite{Sog-sharm}:

\begin{theorem}[\cite{Sog-sharm}]
\label{T:Sogge}
Let $(M,g)$ be a $d$-dimensional compact Riemannian manifold without boundary, and let
$-\Delta$ be the Laplace-Beltrami operator on $M$.  If $\phi_j$ are
the eigenfunctions, 
\[
-\Delta \phi_j = \lambda^2_j \phi_j
\]
with 
$0 = \lambda_1
\leq \lambda_2 \leq \cdots$ the eigenvalues, then
\[
\| \phi_j \|_{L^{2(d+1)/(d-1)}} \leq C \lambda_j^{(d-1)/2(d+1)} \|
\phi_j \|_{L^2}.
\]

\end{theorem}

In particular, for the situation at hand, 
\begin{align*}
\|  v \|_{L^{2^\star}(\reals)L^{2^\star}(\SS^{n-1})} & = \| H_k v \|_{L^{2^\star}(\reals)L^{2^\star}(\SS^{n-1})}
\\
& \leq C k^{(n-2)/2n} \| H_k v \|_{L^{2^\star}(\reals)L^2 (\SS^{n-1})} .
\end{align*}
Now let $\Lambda_k$ be an index set for the $k$th harmonic subspace $\HH_k$ and write
\[
v(t, x, \theta) = \sum_{ l \in \Lambda_k} v_{lk}(t,x) H_{lk}(\theta),
\]
where $H_{lk}$ are the orthonormal spherical harmonics in $\HH_k$.
Now if $p \geq 2$, $2 \leq q \leq 2^*$ ($q < \infty$ if $n = 2$), we have
\begin{align*}
\| v \|_{L^p_T L^q(M)} & \leq Ck^{(n-2)/2n} \| v \|_{L^p_T L^q(\reals)
  L^2(\SS^{n-1})} \\
& \leq C k^{(n-2)/2n} \left\| \left( \sum_{l \in \Lambda_k} |v_{lk}
    |^2 \right)^{1/2} \right\|_{L^p_T L^q(\reals) },
\end{align*}
by Plancherel's theorem.  We further estimate using Minkowski's
inequality repeatedly:
\begin{align*}
\left\| \left( \sum_{l \in \Lambda_k} |v_{lk}
    |^2 \right)^{1/2} \right\|_{L^p_T L^q(\reals) } & = \left(
  \int_0^T \left[\left( \int \left( \sum_l | v_{lk}|^2 \right)^{q/2} dx
  \right)^{2/q} \right]^{p/2} dt \right)^{1/p} \\
& \leq C  \left(
  \int_0^T \left[\sum_l \left( \int  | v_{lk}|^q dx
  \right)^{2/q} \right]^{p/2} dt \right)^{1/p} \\
& = C \left(
  \int_0^T \left[\left(\sum_l  \| v_{lk}\|^2_{L^q} \right)^{1/2}
  \right]^{p} dt \right)^{1/p} \\
& = C \left[ \left(
  \int_0^T \left(\sum_l  \| v_{lk}\|^2_{L^q} \right)^{p/2}
  dt \right)^{2/p} \right]^{1/2} \\
& \leq C \left[ \sum_l \left(
  \int_0^T \| v_{lk}\|^p_{L^q} 
  dt \right)^{2/p} \right]^{1/2} \\
& = C \left( \sum_{l \in \Lambda_k} \| v_{lk} \|_{L^p_T L^q}^2
\right)^{1/2}.
\end{align*}
All told then we have
\begin{equation}
\label{E:str-sum}
\| v \|_{L^p_T L^q(M)} \leq C k^{(n-2)/2n} \left( \sum_{l \in \Lambda_k} \| v_{lk} \|_{L^p_T L^q}^2
\right)^{1/2},
\end{equation}
where $v_{lk}(t,x) = \lll v(t,x,\cdot), H_{lk} \rrr_{L^2(\SS^{n-1})}.$

Using \eqref{E:Pk}, we see that $v_{lk}$ satisfies the equation $(D_t
+ P_k) v_{lk} = 0$, which is a $1$-dimensional Schr\"odinger equation with potential.  We
want to estimate $v_{lk}$ in the $L^2_T L^{2^\star}(\reals)$ norm when
$n \geq 3$, and in $L^p_T L^q$ Strichartz duals with $2 \leq q <
\infty$ in dimension $n = 2$.  
However, since we are now looking at a solution to a one dimensional
Schr\"odinger equation, $2$ and $2^\star$ are not Strichartz duals in $1$ dimension.
The Strichartz dual $p$ to $2^\star = 2n/(n-2)$ ($n \geq 3$) in one dimension
satisfies
\[
\frac{2}{p} + \frac{1}{2^\star} = \frac{1}{2},
\]
or
\[
p = 2n.
\]
We therefore first use H\"older's inequality in $t$, with weights $n$
and $n/(n-1)$ respectively to get
\begin{align*}
\|  v_{lk} \|_{L^2_T L^{2^\star}}^2 & = \int_0^T \| v_{lk}
\|_{L^{2^\star}}^2 dt \\
& \leq T^{(n-1)/n} \|  v_{lk} \|_{L_T^{2n} L^{2^\star}}^2.
\end{align*}
In dimension $n = 2$, we use $p >2$, $2 \leq q < \infty$ and the same
weights in H\"older's inequality to get
\begin{align*}
\|  v_{lk} \|_{L^p_T L^{q}}^2 & = \int_0^T \| v_{lk}
\|_{L^{2^\star}}^2 dt \\
& \leq T^{1/p} \|  v_{lk} \|_{L_T^{2n} L^{2^\star}}^2.
\end{align*}

We have the following proposition.

\begin{proposition}
\label{P:mode-str}
Suppose $v_{lk}$ solves
\[
\begin{cases}
(D_t + P_k) v_{lk} = 0, \\
v_{lk}|_{t=0} = v_{lk}^0 ,
\end{cases}
\]
where $v_{lk}^0 \in H^s$ for some $s >0$.  Then 
for any $T>0$, there exists a constant $C = C_{T}>0 $ such that
\[
\| v_{lk} \|_{L^{2n}_T L^{2^\star}} \leq C \|  \lll k \rrr^\beta  v_{lk}^0 \|_{L^2}.
\]

\end{proposition}
That is, even though $v_{lk}$ solves a Schr\"odiner
equation with a degenerate potential barrier, $v_{lk}$ nevertheless
satisfies Strichartz estimates with an arbitrary $\beta>0$ loss.  As a consequence, we have the
following estimate on natural semiclassical time scales.

\begin{corollary}
\label{C:C1a}
Suppose $v$ solves \eqref{E:Sch-1a} with initial data $v_0 = v_{lk}^0 H_{lk}$.  Then for
$\epsilon>0$ sufficiently small and $T = \epsilon k^{-2/(m+1)}$, $v$
satisfies the Strichartz estimate
\[
\| v \|_{L^2_T L^{2^\star}} \leq C \| \lll k
\rrr^{\eta + \beta} v_0 \|_{L^2},
\]
where
\[
\eta = \frac{1}{2n(m+1)} \left( m(n-2) - n \right).
\]
Moreover, if $v_{lk}$ is a zonal spherical harmonic, this estimate is
near-sharp, in the sense that no {\it polynomial} derivative
improvement is true for every $\beta >0$.
\end{corollary}

\begin{remark}
This corollary shows that on natural semiclassical time scales the
Strichartz estimates are improved.  Indeed, in dimension $n = 2$,
there is a smoothing effect.  The
proof of the near-sharpness of this estimate is in Section
\ref{SS:saturation}.

\end{remark}

\subsection{Proof of Proposition \ref{P:P1a} and Corollary \ref{C:C1a}}

Assuming Proposition \ref{P:mode-str}, we have from \eqref{E:str-sum}
(in dimension $n \geq 3$):
\begin{align*}
\| v \|_{L^2_T L^{2^\star}(M)}^2 & \leq C k^{(n-2)/n}  \sum_{l \in \Lambda_k} \left\|   v_{lk} \right\|_{L^2_T
  L^{2^\star}(\reals)}^2 \\
& \leq C k^{(n-2)/n} T^{(n-1)/n} \sum_{l \in \Lambda_k} \left\| \lll k \rrr^\beta    v_{lk}^0 \right\|_{
  L^{2}(\reals)}^2 \\
& \leq C T^{(n-1)/n} \| \lll k \rrr^{(n-2)/2n + \beta } v_0
\|_{L^2(M)}^2,
\end{align*}
by orthonormality, which is Proposition \ref{P:P1a}, and hence also
Theorem 
\ref{T:T1a}.  A similar computation using \eqref{E:str-sum} holds when
$n = 2$, and $2 \leq q < \infty$.

For Corollary \ref{C:C1a}, the sum is over only one term, and $T \sim
k^{-2/(m+1)}$.  Then in this case, 
\begin{align*}
\| v \|_{L^2_T L^{2^\star}(M)}^2 & \leq C k^{-2(n-1)/n(m+1)} k^{(n-2)/n} \|  v_{lk} \|_{L^{2^\star}(\reals)}^2 \\
& \leq (1 + | k |)^{\frac{1}{n(m+1)} \left( m(n-2) - n \right) } \| \lll k \rrr^\beta  v_{lk}^0 \|_{L^2(M)}^2 \\
& \leq C \| \lll k \rrr^{\eta + \beta} v_0 \|_{L^2(M)}^2,
\end{align*}
where $\eta$ is as in Corollary \ref{C:C1a}.  A similar computation
holds for $q < \infty$ in the case $n = 2$.

\qed




\section{The parametrix}

It remains to prove Proposition \ref{P:mode-str}.  For that purpose, 
in this section we construct a parametrix for the 
separated Schr\"odinger equation:
\[
\begin{cases}
(D_t + (-\partial_x^2 + \lambda_k^2 A^{-2}(x) + V_1(x)) ) u = 0 , \\
u|_{t=0} = u_0.
\end{cases}
\]
We rescale $h^2 = \lambda_k^{-2}$ to get
\[
\begin{cases}
(D_t - (-\partial_x^2 + h^{-2} A^{-2}(x) + V_1(x)) ) u = 0 , \\
u|_{t=0} = u_0.
\end{cases}
\]
Let $v(t, x) = u(ht, x)$, so that
\[
\begin{cases}
(hD_t - (-h^2 \partial_x^2 + A^{-2}(x) + h^2 V_1(x)) ) v = 0 , \\
v|_{t=0} = u_0.
\end{cases}
\]

For the rest of this section, we consider the one-dimensional semiclassical
Schr\"odinger equation 
with barrier potential:
\begin{equation}
\label{E:v-sc-eqn}
\begin{cases}
(hD_t + (-h^2 \partial_x^2 + V(x))) v = 0, \\
v|_{t=0} = v_0.
\end{cases}
\end{equation}
The potential $V(x) = A^{-2}(x) + h^2 V_1(x)$ decays at $| x | =
\infty$, is even, and the principal part $A^{-2}(x)$ has a degenerate maximum at $x = 0$ with no other
critical points.  Denote $P = -h^2 \partial_x^2 + V(x)$.  

Let us give a brief summary of the steps involved in the
construction.  We will use a WKB type approximation, although, since
we are in dimension $1$, we do not need a particularly good
approximation.  The first step is to approximate the solution away
from the critical point at $(x, \xi) = (0,0)$.  Since this is a
non-trapping region, standard techniques can be used to construct a
parametrix and prove Strichartz estimates on a timescale $t \sim
h^{-1}$ for the semiclassical problem, or on a fixed timescale for the
classical problem.  A similar construction applies for energies away
from the trapped set.  

The remaining regions can be divided into an $h$-dependent strongly
trapped region and a ``transition region'', where wave packets
propagate, but not at a uniform rate.  By restricting attention to a
sufficiently small $h$-dependent neighbourhood of $(0,0)$, we can
extend a semiclassical parametrix to a timescale $t \sim
h^{(1-m)/(1+m)}$, which is a classical timescale of $h^{2/(m+1)}$.  We
divide the transition region into a logarithmic number of
$h$-dependent regions on which a similar parametrix construction
works.  Summing over all of these regions gives a parametrix
construction and corresponding Strichartz estimates in a compact
region in phase space with a logarithmic loss due to the number of
summands in the transition region.  These constructions and Strichartz
estimates hold for a frequency dependent timescale $\sim h^{2/(m+1) +
  \beta}$, $\beta>0$, or with a $\beta>0$ loss in derivative on
timescale $\sim h^{2/(m+1) }$.  We
then use the local smoothing estimate from \cite{ChWu-lsm} to glue
estimates on $\sim h^{-2/(m+1)}$ time intervals to get the Strichartz
estimates with a $\beta>0$ loss overall.

\subsection{WKB expansion}

We make the following WKB ansatz:
\[
v = h^{-1/2} \int e^{i \phi(t, x, \xi)/h} e^{-iy \xi} B(t, x, \xi) u_0(y) dy d \xi,
\]
and compute
\[
(hD_x)^2 v = \int e^{i \phi(t, x, \xi)/h} ((\phi_x)^2 B -i h \phi_{xx}
B -2ih \phi_x B_x - h^2 B_{xx} ) u_0 (y) dx d \xi,
\]
and
\[
hD_t v = \int e^{i \phi(t, x, \xi)/h} (\phi_t B -ih B_t) u_0(y) dy d
\xi.
\]
In order to approximately solve the semiclassical Schr\"odinger
equation for $v$, we use the WKB analysis.  We begin by trying to construct $\phi$ so
that
\[
\begin{cases}
\phi_t + (\phi_x)^2 +V(x) = 0, \\
\phi|_{t = 0} = x  \xi.
\end{cases}
\]
Given such $\phi$, we solve the transport equations for the amplitude
using a semiclassical expansion:
\[
B = \sum_{j \geq 0} h^j B_j(t, x, \xi),
\]
and
\[
-ihB_t -i h \phi_{xx}
B -2ih \phi_x B_x - h^2 B_{xx} = 0.
\]
This amounts to solving:
\begin{equation}
\label{E:WKB-amp-h}
\begin{cases}
- B_{0,t} -2 \phi_x B_{0,x} - \phi_{xx} B_0 = 0,
\\
-iB_{j,t} -i \phi_{xx} B_j -2 i  \phi_x B_{j,x} -  B_{j-1,xx} 
= 0, \,\,\, j \geq 1.
\end{cases}
\end{equation}

\subsection{The partition of unity}

In this subsection we construct the partition of unity which will be used
to glue together the parametrices constructed in the
following subsections.  Let $\epsilon>0$, $\delta>0$ be sufficiently
small, and $\omega>1$, all to be specified in the sequel.  Let $\chi
\in \Ci_c( \reals)$, $\chi(r) \equiv 1$ for $|r| \leq 1$, with support
in $\{ | r | \leq 2 \}$ and assume 
$\chi'(r) \leq 0$ for $r \geq 0$.  For $\sigma >0$, let
$\chi_\sigma(r) = \chi(r/\sigma)$.
Let $\chi^\pm
\in \Ci( \reals)$, $\chi^\pm(r) \equiv 1$ for $\pm r \gg 1$, $\chi^\pm
= 0$ for $\pm r \geq 0$, and choose $\chi^\pm$ so that $1 =
\chi(r) + \chi^+(r) + \chi^-(r)$, and denote also $\chi^\pm_\sigma(r) =
\chi^\pm(r/\sigma)$.  
  Choose
also $\psi_0, \psi \in \Ci_c( \reals_+)$ with $\psi_0(r) \equiv 1$
near $r = 0$, and $\psi(r) \equiv 1$ in a
neighbourhood of $r = \delta$ such that 
\[
 \sum_{0}^{N(h)} \psi(\omega^j x ) \equiv 1 \text{ for } x \in
[\delta, 2\epsilon h^{-1/(m+1)} ],
\]
and
\[
\psi_0(x) + \sum_{0}^{N(h)} \left( \psi(\omega^j x )  + \psi(-
   \omega^j x ) \right) \equiv 1 \text{ for } x \in
[-2 \epsilon h^{-1/(m+1)} , 2\epsilon h^{-1/(m+1)} ].
\]
We remark for later use that we take, for example
\[
\psi(\omega^j x) = \begin{cases} 1, \text{ for } \delta (\omega^j +
  \omega^{j-2}) \leq x \leq \delta (\omega^{j+1} - \omega^{j-1} ) \\
  0, \text{ for } x \in [ \delta (\omega^j -
  \omega^{j-2}) , \delta (\omega^{j+1} + \omega^{j-1} ) ]^\complement,
\end{cases}
\]
so that in particular
\[
| \p_x^k \psi(\omega^j x) | \leq C_k (\delta \omega)^k (\omega^{-jk}).
\]
We also observe this implies we need $N(h)$ sufficiently large that
$\omega^j \sim h^{-1/(m+1)}$, so that $N(h) = \O( \log(1/h))$, with
constants depending on $\delta, \omega$, and $m$.

We write
\[
e^{itP/h} = L(t) + S(t) := (1 - \chi_\epsilon(x)) e^{itP/h} +
\chi_\epsilon(x) e^{itP/h}
\]
for the propagator cut off to large and small values of $x$
respectively.  The set where the symbol $p = 1$ contains the critical
point $(0,0)$, so we further decompose into frequencies $\xi$ which
lie above (respectively below) the set where $p = 1$, and frequencies
which are bounded:
\[
S(t) = \Shi(t) + \Slo(t) ,
\]
where
\[
\Shi(t) = \11_{\{\pm hD_x \geq 1-V(x)\} } (1-\chi_{\epsilon^2} ((P-1))) S(t) ,
\]
and $\Slo(t) = S(t) - \Shi(t)$.
We
decompose yet again to
\begin{align*}
\Slo(t) & = \Sloo(t) + \sum_{j=0}^{N(h)} (\Slojp(t) + \Slojm(t) ) ,
\end{align*}
where
\[
\Sloo(t) = \psi_0(x/h^{1/(m+1)}) \Slo(t),  
\]
and
\[
\Slojpm(t) = \psi( \pm  \omega^j x/h^{1/(m+1)}
 )\Slo(t).
\]

The operators $\Slojpm(t)$ are localized to bounded frequencies, and
dyadic strips of size $h^{1/(m+1)} \omega^j$.  We require one further
localization, which is to assume that the operators are also
outgoing/incoming.  Choose $\tchi \in \Ci( \reals)$ so that $\tchi(r)
= 1$ for $r \geq 1$ and $\tchi(r) = 0$ for $r \leq 0$.  
For $a, \gamma >0$ to be determined, let 
\[
\Slojp^\pm(t) = \tchi((\pm hD_x + ax^m)/\gamma x^m) \Slojp(t),
\]
and
\[
\Slojm^\pm(t) = \tchi((\mp hD_x + ax^m)/\gamma x^m) \Slojm(t).
\]
This has the effect of localizing in phase space to the sets where 
\[
\pm \xi \geq -ax^m
\]
for $\Slojp^\pm(t)$ and similarly for $\Slojm^\pm(t)$.  By the
properties of $\tchi$, we have
\[
\Slojp^\pm(t) = \Slojp(t)
\]
microlocally on the set
\[
\{ \pm \xi \geq (\gamma -a) x^m \}.
\]
If $a > \gamma$, these two sets clearly cover the remaining phase
space, so if we can estimate each one of the operators above, we have
estimated the entire propagator.



It is clear then that if we can prove, say, $\beta/2>0$ loss Strichartz estimates for
$\Sloo(t)$, and for each $\Slojp^+$ and $\Slojm^-(t)$ for $t \geq 0$, the Strichartz estimates follow
for $\Slojp^-$ and $\Slojm^+(t)$ by time reversal.  We thus have to prove Strichartz
estimates for each of these operators, as well as for $\Shi(t)$ and $L(t)$, at which
point we can sum up and take a loss of $\log(1/h) + h^{-\beta/2} < C h^{-\beta}$.

\subsection{The parametrix for $L(t)$}
We recall that the operator $L(t)$ is the propagator localized to
large $| x |$.  
Then the operator $L(t)$ can be decomposed into $L^+(t) + L^-(t)$,
supported where $\pm x >0$ respectively.  Thus
\[
L^+(t) = \chi^+_\epsilon(x) e^{itP/h}.
\]
By a $T T^*$ argument (see \cite{KT}), in order to show $L^+ : L^2
\to L^p_T L^q$, it suffices to estimate 
\[
L^+(t) (L^+)^*(s) : L^1 \to L^\infty,
\]
but
\[
L^+(t) (L^+)^*(s) = \chi^+_\epsilon(x) e^{i(t-s)P/h}
\chi^+_\epsilon(x).
\]
That is, we need only construct a parametrix supported for $x \geq
\epsilon$, and for initial data supported for $x \geq \epsilon$.  

\begin{lemma}
\label{L:L}
There exist constants $C>0$ and $\alpha>0$ such that for any $u_0 \in
L^1 \cap L^2$, we have
\[
\| L^+(t) (L^+)^*(s) u_0 \|_{L^\infty_x} \leq C (|t-s|h)^{-1/2} \| u_0
\|_{L^1},
\]
for $| t | , | s | \leq \alpha h^{-1}$.  
As a consequence, 
\[
\| L(t) u_0 \|_{L^{p}_{\alpha h^{-1} } L^{q}} \leq h^{-1/p} \| u_0 \|_{L^2}
\]
for 
\[
\frac{2}{p}+ \frac{1}{q} = \frac{1}{2}, \,\,\, 2 \leq q < \infty.
\]

\end{lemma}

\begin{proof}

The proof is simply to observe that $L^+(t) (L^+)^*(s)$ is equal to a
non-trapping cut-off propagator, and hence obeys a strong dispersion
and perfect Strichartz estimates according to \cite{BoTz-gstr}.

To see this, let 
\[
\tA(x)^{-2} = \chi(x/\epsilon) x^{-2} + (1 - \chi(x/\epsilon)) A^{-2}.
\]
The function $\tA$ agrees with $A$ for large $x$ and agrees with $x$
for small $x$.  Then
\[
\tg = dx^2 + \tA^2(x) d \theta^2, \,\,\, x \geq 0
\]
is an asymptotically Euclidean metric, which agrees with the Euclidean
metric near $x = 0$.  In fact, since $g$ was a short-range
perturbation of the Euclidean metric as $x \to + \infty$, so is
$\tg$.  In addition, we claim that for $\epsilon>0$ sufficiently
small, $\tg$ is a {\it non-trapping} perturbation of the Euclidean
metric.  To see this, we examine the geodesic equations.  Let $\tp =
\xi^2 + \tA^{-2}(x) \eta^2$, and compute the geodesic equations:
\[
\begin{cases}
\dot{x} = 2 \xi, \\
\dot{\xi} = 2 \tA' \tA^{-3} \eta^2, \\
\dot{\theta} = 2 \tA \eta, \\
\dot{\eta} = 0.
\end{cases}
\]
Consider a unit speed geodesic with $\tp \equiv 1$.  Since $\eta$
remains constant, then either $\eta = 0$, in which case $\xi = \pm 1$
and $x \to \pm \infty$ uniformly, or $\eta \neq 0$.  If $\xi \equiv
0$, then necessarily $(\tA^{-2}(x))' = 0$ and $x$ is stationary, but
\[
(\tA^{-2}(x))'  = -2 \chi(x/\epsilon) x^{-3} -2 (1 - \chi(x/\epsilon))
A' A^{-3} + \epsilon^{-1} \chi'(x\epsilon) ( x^{-2} - A^{-2}).
\]
But $A' >0$ away from $x = 0$, $x^{-2} \gg A^{-2}$ for $x>0$
sufficiently small, and $\chi' \leq 0$ for $x>0$ implies
$(\tA^{-2}(x))' <0$ for $x >0$.  Hence there are no parallel periodic
geodesics.  

It remains to show that every other trajectory escapes to infinity.
But since $\dot{\xi} \geq c^{-1} x^{-3} \eta^2$, comparing to the
non-trapping conic metric with 
\[
\begin{cases}
\dot{x} = 2 \xi, \\
\dot{\xi} = c^{-1} x^{-3} \eta^2
\end{cases}
\]
implies that every other trajectory is non-trapped.
Then following Bouclet-Tzvetkov \cite{BoTz-gstr}, we get that 
\[
L^+ : L^2 \to L^p_{\alpha h^{-1}} L^q , \,\,\, \alpha >0,
\]
is a bounded operator for $(p,q)$ in the specified range.  
A similar estimate holds for $L^-$, and hence $L$, and hence for any $\epsilon>0$ sufficiently small, we can construct a
parametrix to get perfect Strichartz estimates for $| x | \geq \epsilon$.

\end{proof}


\subsection{The parametrix for $\Shi(t)$}
The operator $\Shi(t)$ is the propagator localized to small $| x |
\leq 2\epsilon$ and high frequencies $| P - 1 | \geq \epsilon^2$, and
$\pm \xi \geq 1-V(x)$.  
In order to estimate $\Shi(t)$, we employ a similar argument.  We
first decompose $\Shi(t) = \Shi^+(t) + \Shi^-(t)$ into a part supported in $\pm \xi >0$.   The point of the next lemma is that
singularities propagate out of this region quickly, depending on the
initial frequency.

\begin{lemma}
\label{L:Shi}
There exist  constants $\alpha, \kappa>0$ such that 
\[
\chi^+(|t-s| hD_x / \kappa \epsilon )  \Shi^+(t) (\Shi^+)^*(s) =
\O(h^\infty)
\]
in any seminorm, provided $| t | , | s | \leq \alpha h^{-1}$

There exist constants $C>0$ and $\alpha>0$ such that for any $u_0 \in
L^1 \cap L^2$, we have
\[
\| \Shi^+(t) (\Shi^+)^*(s) u_0 \|_{L^\infty_x} \leq C (|t-s|h)^{-1/2} \| u_0
\|_{L^1},
\]
for $| t|, |s | \leq \alpha h^{-1}$.  
As a consequence, 
\[
\| \Shi(t) u_0 \|_{L^p_{\alpha h^{-1} } L^q} \leq C h^{-1/p} \| u_0 \|_{L^2}
\]
for 
\[
\frac{2}{p}+ \frac{1}{q} = \frac{1}{2}, \,\,\, 2 \leq q < \infty.
\]


\end{lemma}


\begin{proof}

As usual, we consider the Hamiltonian system associated to $p$:
\[
\begin{cases}
\dot{x} = 2 \xi, \\
\dot{\xi} = -V'(x), \\
x(0) = y, \\
\xi(0) = \eta,
\end{cases}
\]
where now $|x |, | y | \leq 2\epsilon$ and $\eta \geq \epsilon$.  Then
a simple computation shows that in this region $|V'(x) | =
\O(\epsilon^{2m-1})$ and $| V''(x) | = \O( \epsilon^{2m-2})$.  Hence
if $t = \O(1)$, we have
\[
\xi = \eta + \O( \epsilon^{2m-1}) = \eta(1 + \O(\epsilon^{2m-2})),
\]
since $\eta \geq \epsilon$.  Hence
\[
\dot{x} = 2\eta(1 + \O(\epsilon^{2m-2})),
\]
so that
\[
x = y + 2 t \eta(1 + \O(\epsilon^{2m-2})),
\]
provided $t = \O(1)$.  This implies in particular, that for any $| t |
\geq C \epsilon / \eta$, we will have $| x | \geq 2\epsilon$, so that
we have propagated out of the region of interest.  Again, by virtue of
a $T T^*$ argument, we are interested in both initial data and
parametrix localized in $| x | \leq 2\epsilon, \xi \geq \epsilon$, so
we need only check the estimates on the phase function for $| t | \leq
C \epsilon / \eta$.  

We check the invertibility of the map $y \mapsto x(t)$:
\begin{align*}
\sup_{| t | \leq C \epsilon / \eta } \left| \frac{\partial
    x}{\partial y} (t) \right| & \leq 1 + 2 \int_0^{C \epsilon / \eta
} (C \epsilon / \eta - s) | V''(x)| \left| \frac{\partial
    x}{\partial y} (s) \right| ds \\
& \leq 1 + \O( \epsilon^2 / \eta^2 ) \O( \epsilon^{2m-2} ) \sup_{| s |
  \leq C \epsilon / \eta }\left| \frac{\partial
    x}{\partial y} (s) \right| ,
\end{align*}
which implies
\begin{align*}
\sup_{| t | \leq C \epsilon / \eta } \left| \frac{\partial
    x}{\partial y} (t) \right| & \leq 1 +\O( \epsilon^{2m-2} ) .
\end{align*}
Similarly we compute the lower bound:
\begin{align*}
\inf_{| t | \leq C \epsilon / \eta } \left| \frac{\partial
    x}{\partial y} (t) \right| & \geq 1 - 2 \int_0^{C \epsilon / \eta
} (C \epsilon / \eta - s) | V''(x)| \left| \frac{\partial
    x}{\partial y} (s) \right| ds \\
& \geq 1 - \O( \epsilon^2 / \eta^2 ) \O( \epsilon^{2m-2} ) \sup_{| s |
  \leq C \epsilon / \eta }\left| \frac{\partial
    x}{\partial y} (s) \right| \\
& \geq 1 - \O( \epsilon^{2m-2}),
\end{align*}
using our previously computed upper bound.  Hence in the range in
which we are interested, $\partial x / \partial y$ is uniformly
bounded above and below by a constant, provided $\epsilon>0$ is chosen
sufficiently small.

It is now a routine computation to construct the WKB amplitude and
compute the dispersive estimate for $| t | \leq C \epsilon/ \eta$.
After that time, the $h$-wavefront set of a solution is outside the
support of the cutoffs in $\Shi(t)$, so that any parametrix
approximation is $\O(h^\infty)$.  Summing over $\O(h^{-1})$ such
parametrices yields the dispersive estimate for $| t | \leq \alpha
h^{-1}$, and the associated Strichartz estimates.

A similar computation works for $\Shi(t)$ localized to $\xi \leq -
\epsilon$, which proves the lemma for $\Shi(t)$.

\end{proof}

\subsection{The parametrix for $\Sloo(t)$}
\label{SS:sloo}
The operator $\Sloo(t)$ is the propagator localized to small
frequencies $| P -1| \leq \epsilon^2$ or $|P-1| \geq \epsilon^2$ with
$| \xi | \leq 1 - V(x)$, as well as localized to a small $h$-dependent
spatial neighbourhood $| x | \leq \delta h^{1/(m+1)}$.  This is the
region which contains the trapping.  We observe that all of the $\Slo$
operators have $| x | \leq \epsilon$, which implies in addition that $| \xi | \leq 2
\epsilon$, say.

 We are now interested in
constructing a parametrix in the set $\{ | x | \leq \delta h^{1/(m+1)}, |
\xi | \leq \epsilon \}$.  
For this, we use the following
$h$-dependent scaling operator:
\[
T_h u(t,x) = h^{-1/(m+1)} u(h^{(m-1)/(m+1)} t , h^{-1/(m+1)}x).
\]
The purpose of the prefactor of $h^{-1/(m+1)}$, different from the
usual scaling prefactor, is to ensure that $\| T_h u \|_{L^1_x} =
\| u \|_{L^1_x}$, since in our final dispersion estimate, this is how
the initial data will be measured.  We compute:
\begin{align*}
T_h^{-1} (hD_t -h^2 \partial_x^2 + V(x) ) T_h & = (h^{(m-1)/(m+1)} hD_t
- h^{-2/(m+1)} h^2 \partial_x^2 + V(h^{1/(m+1)} x) ) \\
& = h^{2m/(m+1)} ( D_t - \partial_x^2 + \tV(x;h)),
\end{align*}
where
\[
\tV(x;h) = h^{-2m/(m+1)} V(h^{1/(m+1)} x).
\]

\begin{remark}
Similar to the the techniques in the paper \cite{ChWu-lsm},
conjugation by the scaling operator $T_h$ is an inhomogeneous
``blowup'' procedure.  However, the blowdown map $\B$ is now {\it
  time-dependent} and takes the form
\[
\B(t, \tau, x, \xi) = ( h^{(1-m)/(m+1)} t, h^{2m/(m+1)} \tau,
  h^{1/(m+1)} x, h^{m/(m+1)} \xi ).
\]
That is, we are blowing up the $(\tau, x, \xi)$ coordinates and
blowing down the $t$ coordinate at the same time.  Observe that the
blowdown in $t$ does not
cause a problem with the calculus since the operator $P$ is
independent of $t$.  Then indeed, according
to the calculus developed in \cite{ChWu-lsm}, $\sigma_h(P) = \tau +
\xi^2 + V(x)$ in the $h$ calculus, while $T_h^{-1} P T_h$ has symbol
\[
\tp_1 = (h^{2m/(m+1)} \tau ) + ( h^{m/(m+1)} \xi )^2 + V(h^{1/(m+1)}x)
\]
in the $1$-calculus, or scale-invariant calculus.  
Factoring out the
$h^{2m/(m+1)}$ as above results in a singular symbol in the
scale-invariant calculus (see Figure \ref{fig:phase-blowup}).  However, the special structure of $V$
allows us to construct a reasonable parametrix where $V'$ is extremely
small, and where $V'$ is large, wave packets propagate away in a
controlled fashion.  This is made
rigorous in the following constructions.  
\end{remark}

Denote $\tP = D_t - \partial_x^2 + \tV(x;h)$, where 
\[
\tV(x;h) = h^{-2m/(m+1)} V(h^{1/(m+1)} x).  
\]
We break the parametrix
construction into two sets, where $\tV'$ is small (and hence this
region contains the trapping), and where $\tV'$ is
large, which we reserve for the next subsections where we estimate $\Slojpm^\pm(t)$.

We want to now
construct a parametrix for $\tP$ on the set 
\[
\{ | x | \leq \delta, | \xi |
\leq 2 h^{-m/(m+1)}, | t | \leq 1 \},
\]
but in the $1$-calculus (scale-invariant).  Then if $w(t,x)$ is such
a parametrix, $v(t,x) = T_h w(t,x)$ is a parametrix for $P$ on the set
$\{ | x | \leq \delta h^{1/(m+1)}, | \xi | \leq \epsilon, | t | \leq
h^{(1-m)/(1+m)} \}$, as required.

\begin{figure}
\hfill
\centerline{\input{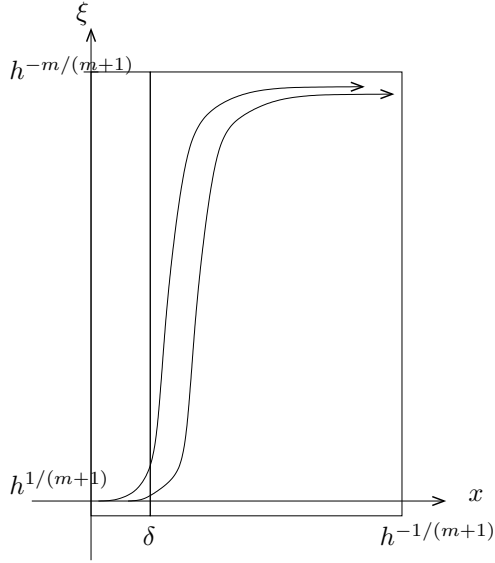}}
\caption{\label{fig:phase-blowup} The phase plane in the blown up
  coordinates.  The invariant curves are given by
  $L^\zeta = \{\xi  = \sqrt{\zeta^2 - \tV(x) }$, with $\zeta^2 \geq
  h^{-2m/(m+1)}$.  The boxes represent the $h$-wavefront set of a wave
packet after rescaling $(x, \xi) \mapsto (h^{-1/(m+1)}x, h^{-m/(m+1)}
\xi)$, but in the $1$-calculus.}
\hfill
\end{figure}

\begin{lemma}
\label{L:Slooo}
There exists $\alpha>0$ and a phase function $\phi(t, x, \eta)$ satisfying
\[
\begin{cases}
\phi_t + \phi_x^2 + \tV(x;h) = 0, \\
\phi(0, x, \eta) = x \eta
\end{cases}
\]
for $| x | \leq \delta$, $| \eta | \leq 2\epsilon h^{-m/(m+1)}$, and $|
t | \leq \alpha$.

We further have 
\[
\phi_{\eta \eta} \sim 2 t(1 + \O(t)),
\]
and
\[
\phi_{xx} = \O(tx^{2m-2})
\]
for $| t | \leq \alpha$.
\end{lemma}

\begin{proof}
The proof is by the usual Hamiltonian method.  We consider $q = \xi^2
+ \tV(x;h)$ and the Hamiltonian system associated to $q$:
\begin{equation}
\label{E:Ham-q}
\begin{cases}
\dot{x} = 2 \xi, \\
\dot{\xi} = - \partial_x \tV(x;h), \\
x(0) = y, \\
\xi(0) = \eta.
\end{cases}
\end{equation}
Now the potential $\tV(x;h)$ has been computed above, and satisfies
\begin{align*}
-\partial_x \tV(x;h) & = -\partial_x \left(  h^{-2m/(m+1)}(1 +
(h^{1/(m+1)}x)^{2m} )^{-1/m} + h^{2/(m+1)} V_1(h^{1/(m+1)} x) \right) \\
& = 2h^{-2m/(m+1)} h^{1/(m+1)} (h^{1/(m+1)}x)^{2m-1} (1 +
(h^{1/(m+1)}x)^{2m} )^{-1/m -1} \\
& \quad + \O( h^{3/(m+1)}
((h^{1/(m+1)}x)^{2m-3}   ).
\end{align*}
For $| x | \leq \delta$ this derivative is bounded, and has the same
sign as $x$.  Let us denote $B(x) = - \partial_x \tV(x;h)$ to avoid
cumbersome notation.  

In order to apply the usual Hamilton-Jacobi theory, we need to show
that $\partial x/ \partial y$ is uniformly bounded above and below by
positive constants on some interval $| t | \leq \alpha$, so that we
can invert the transformation $y \mapsto x(t)$ to get $y = y(t,x)$.
Then using $(x, \eta)$ as coordinates instead of $(y, \eta)$ proves
the first part of the Lemma.  We write
\[
x(t) = y + 2 t \eta + \int_0^t (t-s) B(x(s)) ds,
\]
and compute
\[
\frac{\partial x}{\partial y}(t) = 1 + \int_0^t(t-s) B'(x(s))
\frac{\partial x}{\partial y}(s)ds.
\]
We know 
\[
\frac{\partial x}{\partial y}(0) = 1,
\]
and $B'(x)$ is non-negative in a neighbourhood of $x = 0$, so the
integral in the above expression is positive for $|x| \leq \delta$ and
$| t | \leq \alpha$ sufficiently small.  Further, $B'$ is bounded for
$|x | \leq \delta$, so the integral expression is also bounded above
for $| t | \leq \alpha$.  Hence by restricting $|x|$ and $| t |$ to
fixed, bounded ranges, we conclude the map sending $y \mapsto x(t)$ is
invertible, and this completes the proof of the first assertion.

We observe that, by construction, $\phi_\eta (t, x, \eta) = y$, so
that to compute $\phi_{\eta \eta}$, we need to compute
\[
\frac{\partial y}{\partial \eta} = \frac{ \partial y}{\partial x}
\frac{ \partial x}{\partial \eta}.
\]
We have already shown that $\partial y / \partial x$ is bounded above
and below for $| t | \leq \alpha$, so we compute
\begin{align*}
\frac{ \partial x}{\partial \eta} & = 2t + \int_0^t(t-s) B'(x(s))
\frac{\partial x}{\partial \eta}(s)ds \\
& = 2t + \O(t^2) \sup \frac{\partial x}{\partial \eta}.
\end{align*}
This implies 
\[
\sup_{|t| \leq \alpha} \frac{\partial x}{\partial \eta}(t) \leq 2t(1 +
\O(t)).
\]
Plugging this into the integral expression above yields
\[
\inf_{|t| \leq \alpha} \frac{\partial x}{\partial \eta} \geq 2t(1 +
\O(t^3)).
\]

Finally, since the intertwining relation gives $\phi_x (t, x, \eta) =
\xi$, we have
\[
\phi_{xx} = \p_y \xi \p_x y
\]
in the notation above.  We have already shown that $\p_x y$ is bounded
above and below by a positive constant for $| t | \leq \alpha$, so we
just need to compute
\begin{align*}
\p_y \xi & = \p_y \left( \eta - \int_0^t \tV'(x(s))  ds \right) \\
& = -\int_0^t \tV''(x(s)) \p_y x(s) ds \\
& = \O( t x^{2m-2} ).
\end{align*}
This is the last assertion in the Lemma.

\end{proof}

We now construct the amplitude for the parametrix for
the operator $\Sloo(t)$.  This, combined with
Lemma \ref{L:Slooo}, will be used to compute a
dispersion estimate, resulting in a
Strichartz estimate.  The problem is that, since we are working in a
marginal calculus, the error terms in our parametrix are just too
large.  For example, the error term $\phi_{xx} \sim t x^{2m-2}$
computed in Lemma \ref{L:Slooo} rescales as
\[
T_h \phi_{xx} \sim h^{(m-1)/(m+1)} t h^{-(2m-2)/(m+1)} x^{2m-2}
\sim h^{(1-m)/(1+m)} t x^{2m-2}.
\]
This operator, when composed with the appropriate oscillatory
integral, yields an $L^2$ bounded operator for each $t$, $| t | \leq h^{(1-m)/(m+1)}$.  However, to apply an
energy estimate or a local smoothing estimate, we either have to
integrate in time (now an interval of length $\sim h^{(1-m)/(1+m)}$),
or pull out a factor of $x^{m-1}$ to apply Theorem \ref{T:lsm}.  In
either case, we lose a factor of $h^{(1-m)/2(m+1)}$.  Hence at this
point we must accept an $\beta >0$ loss in regularity by restricting
our attention to a slightly smaller time interval.  Then the ``lower
order'' terms in the amplitude construction will actually gain powers
of $h$.

We are interested in constructing a parametrix for
the operator $\Sloo(t) \Sloo^*(s)$.  We have constructed a phase
function $\phi(t, x, \xi)$ in rescaled coordinates, assuming
appropriate microlocal cutoffs.  That is, we have constructed the
appropriate phase functions to approximate the operators
\[
T_h^{-1} \Sloo(t) \Sloo^*(s) = T_h^{-1} \Sloo(t-s) \chi_\star,
\]
where $\chi_\star$ is the appropriate microlocal cutoff.  
We have not yet computed the amplitude.  
Recalling the transport equations in the $h$-calculus \eqref{E:WKB-amp-h}, the transport
equations for the amplitude $B$ in the rescaled $1$-calculus
coordinates become
\[
D_t B + 2 \phi_x D_x B -i \phi_{xx} B - \partial_x^2 B = 0.
\]
The standard technique here is to guess an asymptotic series, however,
there is no small parameter, so we instead modify our ansatz to take
advantage of the Frobenius theorem.  

That is, the Frobenius theorem guarantees the existence of a function
$\Gamma(t,x),$ depending implicitly on the frequency $\xi$, satisfying
\[
\begin{cases}
\partial_t \Gamma + 2 \phi_x \partial_x \Gamma = 0, \\
\Gamma(0,x) = x.
\end{cases}
\]
We then construct $B = \sum_{j = 0}^KB_j$ for sufficiently large $K$ to
be determined (independent of $h$) with 
\[
\begin{cases}
B_0 \equiv 1, \\
B_j = -\int_0^t \phi_{xx} B_{j-1} |_{(s, \Gamma(t-s,x))} + i
B_{j-1,xx} |_{(s, \Gamma(t-s,x))}.
\end{cases}
\]
An induction argument shows that $B_j = \O(t^j)$ for each $j$, since
we are in the scale invariant calculus.



Then 
\[
w(t, x) = (2 \pi)^{-1} \int e^{i \phi(t, x, \xi) - i y \xi }
B(t, x, \xi) \chi_\star (y, D_y)^* w_0(y) dy d \xi
\]
solves
\[
\begin{cases}
\tP w = \tE, \\
w(0, x) = \chi_\star (x, D_x)^* w_0(x),
\end{cases}
\]
where 
\[
\chi_\star = T^{-1}_h \psi_0(x/h^{1/(m+1)}) ( 1 - \11_{\{\pm hD_x \geq
   1-V(x)\} } (1-\chi_{\epsilon^2} ((P-1))) ) \chi_\epsilon(x) T_h
\]
is the appropriate microlocal cutoff, and the equation is
understood to make sense for $|t| \leq \alpha$.  
Here, the error $\tE$ is given by
\[
\tE  = (2 \pi)^{-1} \int e^{i \phi(t, x, \xi) - i y
  \xi} (-\partial_x^2 B_K -i \phi_{xx} B_K )\chi^*(y,
D_y) w_0(y) dy d \xi .
\]
That is, $\tE$ is an oscillatory integral operator with the same phase
function as $w$, and amplitude $A(t,x,\xi)$ satisfying 
\[
| \partial_x^k \partial_\xi^l A| \leq C_{kl} t^K,
\]
and hence, according to the next Lemma, satisfies
\[
\| \tE \|_{L^2_x} = \O (t^K) \| \chi^* w_0 \|_{L^2}.
\]

\begin{lemma}
Suppose $\Gamma(t, x, \xi) \in \Ci_{b} \s_{0,0}$ is a smooth family of symbols with bounded derivatives,
and let $F(t)$, $0 \leq t \leq \alpha$ be the operator defined by
\[
F(t) g(x) = \int e^{i \phi(t, x, \xi) - i y \xi } \Gamma(t, x, \xi)
\chi_\star(y, D_y)
g(y) dy d \xi,
\]
where $\phi$ is the phase function constructed above and
$\chi_\star$ is the appropriate microlocal cutoff.  Then 
\[
\sup_{0 \leq t \leq \alpha} \| F(t) g \|_{L^2} \leq C \| g \|_{L^2}.
\]
\end{lemma}

\begin{proof}

Let us work microlocally to avoid continually using microlocal
cutoffs, and therefore assume the appropriate microlocal
concentration.  The $L^2$ boundedness of $F(t)$ is equivalent to the
$L^2$ boundedness of $F(t)^*$, which follows from the $L^2$
boundedness of $F(t) F(t)^*$.  The operator $F(t) F(t)^*$ is easily
seen to have integral kernel
\[
K = \int e^{i \phi(t, x, \xi) - i \phi(t, x', \xi)} \Gamma(t, x, \xi)
\bar{\Gamma}(t, x', \xi) d \xi,
\]
where again we are implicitly assuming appropriate microlocal cutoffs.

By stationary phase, this integral kernel has singularities when
\[
\partial_\xi (\phi(t, x, \xi) -  \phi(t, x', \xi)) = 0,
\]
which is when (using the notations from the phase construction)
\[
y(t, x, \xi) - y(t, x', \xi) = 0.
\]
Let us assume that $x \geq x'$, so that we want to compute where
\[
(x - x') \left( \partial_x y |_x + \O( \partial_x^2 y (x - x') )
\right).
\]
Now due to the microlocal cutoffs $\chi_\star$, we have that $x$
and $x'$ are both small.  By the
inverse function theorem and the boundedness of $\partial_x y$, we
need to estimate $\partial_y^2 x$ in the Hamiltonian systems used to
construct the phase functions.  We compute
\[
\partial_y^2 x = - \int_{0}^t (\tV'''(x) (\partial_y x)^2 +
\tV''(x) \partial_y^2 x) ds,
\]
and estimating the first term by $c$ for a small constant $c$
and solving for $\sup \partial_y^2 x$ shows that
\[
|\O(\partial_x^2 y (x - x') ) | \leq c',
\]
where $c'>0$ is a small constant depending on our previous choices of
$\epsilon$, $\delta$, and $\omega$.

Iterating this argument for other powers of $(x-x')$ shows that the singularities of the integral kernel lie on the diagonal $|
x - x' | = 0$, so the integral kernel defines a $0$ order
pseudodifferential operator with symbol in the class $\s_{0,0}$.  By
the Calder\'on-Vaillancourt theorem, the $L^2$ boundedness is
established.

\end{proof}


If we now take $v = T_h w$, we see
\begin{align*}
P v & = T_h T_h^{-1} P T_h w \\
& = h^{2m/(m+1)} T_h \tP w \\
& = E,
\end{align*}
with initial conditions
\[
v(0, x) = T_h w(0,x),
\]
and where
\[
E = h^{2m/(m+1)} T_h \tE.
\]
A simple computation shows that $\| T_h f \|_{L^2} = h^{-1/2(m+1)} \|
f \|_{L^2}$, so that if we now restrict attention to the smaller (rescaled)
time interval
\[
0 \leq t \leq \alpha h^{(1-m)/(m+1) + \beta}
\]
for some small fixed $\beta >0$, we have
\begin{align*}
\sup_{0 \leq t \leq \alpha h^{(1-m)/(m+1) + \beta}  }\| E \|_{L^2} & =
h^{(4m-1)/2(m+1)}  \sup_{0 \leq t \leq \alpha h^\beta } \| \tE \|_{L^2} \\
& \leq C h^{(4m-1)/2(m+1)} h^{\beta K} \| \chi_\star^*w_0 \|_{L^2} \\
& \leq C h^{2m/(m+1)} h^{\beta K}\| \chi_\star^* v_0 \|_{L^2}.
\end{align*}
Here, in the above computations, we have suppressed the variables of
the microlocal cutoffs $\chi_\star$, which are understood to be evaluated
in the phase space variables of the appropriate scale.

The following lemma contains the dispersion and Strichartz estimates
for the operators $\Sloo(t)$.

\begin{lemma}
\label{L:disp-sloo}

The parametrix $v(t, x)$ satisfies the dispersion estimate
\[
\| \chi_\star v \|_{L^\infty} \leq C (ht )^{-1/2}
\| \tchi v_{0} \|_{L^1},
\]
where $0< t \leq \alpha h^{(1-m)/(1+m)} $, as well as the
corresponding Strichartz estimate
\[
\| v \|_{L^p_{\alpha h^{(1-m)/(1+m)} } L^q} \leq C
h^{-1/p} \|
\chi v_{0} \|_{L^2},
\]
for 
\[
\frac{2}{p} + \frac{1}{q} = \frac{1}{2}, \,\,\, q < \infty,
\]
and constants independent of $h$.

The cutoff propagator $\Sloo$ satisfies 
\[
\| \Sloo \|_{L^2 \to L^{p}_{\alpha h^{(1-m)/(1+m) + \beta}} L^{q}
}\leq C h^{-1/p} ,
\]
and
\[
\| \Sloo \|_{L^2 \to L^{p}_{\alpha h^{(1-m)/(1+m) }} L^{q}
}\leq C h^{-(1+\beta)/p} ,
\]
for $(p,q)$ in the same range and constants independent of $h$.

\end{lemma}

\begin{remark}
Observe that the parametrix satisfies good Strichartz estimates all
the way up to the critical time scale $t \sim h^{(1-m)/(m+1)}$, but we
are only able to conclude that the propagator obeys perfect Strichartz
estimates on a slightly shorter time scale, or obeys Strichartz
estimates with a small loss on the critical time scale.  This is an artifact of
working in the marginal calculus and trying to make error terms small
in $h$.
\end{remark}


\begin{proof}
We have
\begin{align*}
v & (t, x) \\
& = T_h w(t, x) \\
& = T_h (2 \pi )^{-1} \int e^{i \phi(t, x, \xi)
  - i y \xi } B(t, x, \xi) \chi_\star^* (y, D_y, h) w_0(y) dy d \xi \\
& =  h^{-1/(m+1)} (2
\pi )^{-1}  \int e^{i \phi(h^{(m-1)/(m+1)}t, h^{-1/(m+1)} x, \xi)
  } \\
& \quad \cdot e^{-iy\xi} B( h^{(m-1)/(m+1)}t, h^{-1/(m+1)} x, \xi
  )\chi_\star^* (y, D_y, h) w_0(y) dy d \xi \\
& = (2
\pi h )^{-1} \int e^{i \phi_\star (t, x, \xi)
  - i y \xi /h } B_{\star}(t, x, \xi ) T_h \chi_\star^* (y, D_y, h) w_0(y) dy d \xi,
\end{align*}
where we use the notation
\[
\phi_\star(t, x, \xi) = \phi(h^{(m-1)/(m+1)}t, h^{-1/(m+1)} x,
    h^{-m/(m+1)} \xi),
\]
and similarly for $B$.
We rewrite this expression as
\[
v_\star(t, x) = \int_y K_\star(t, x, y; h) \chi_\star v_{\star,0}(y)
dy,
\]
where
\[
K_\star(t, x, y;h) = (2 \pi h)^{-1} \int e^{i \phi_\star(t, x, \xi) - iy \xi/h}
B_{\star,0}(t, x, \xi) \tchi_\star(y, \xi; h) d \xi,
\]
 and
\[
\chi_\star  v_{\star,0}(y) = T_h \chi_\star^* (y, D_y, h) w_0(y).
\]
We have already computed the derivative properties of the functions
$\phi$ and $B$ in order to apply the lemma of stationary phase (with
$h$ as small parameter).  The unique critical point is at 
\[
\partial_\xi ( h \phi_\star(t, x, \xi) - y \xi) = 0,
\]
so the
leading asymptotic is
\begin{align*}
(2 \pi h)^{-1/2} & | \partial_\xi^2 (h \phi_\star (t, x, \xi) -y \xi )
|^{-1/2} \\
& = (2 \pi h)^{-1/2} | h
h^{-2m/(m+1)} \phi_{\xi \xi}|_{ (h^{(m-1)/(m+1)}t, h^{-1/(m+1)} x,
    h^{-m/(m+1)} \xi)} \\
& \sim h^{-1/2} | h
h^{-2m/(m+1)} h^{(m-1)/(m+1)}t |^{-1/2} \\
& = | h t|^{-1/2},
\end{align*}
as claimed.  The Strichartz estimates follow immediately.

We now estimate the difference between the propagator and the
parametrix in the $L^\infty_x$ norm to prove that the actual
propagator has the correct dispersion, at least on a slightly shorter
time scale.  Let $u(t,x) = \Sloo(t) v_0(x)$, so that
\[
\begin{cases}
(hD_t + P ) (v-u) = E, \\
(v-u)|_{t = 0 } = 0.
\end{cases}
\]
Since the propagator and the parametrix are compactly
essentially supported in frequency on scale $h^{-1}$ we have the
endpoint Sobolev embeddings:
\[
\sup_{|t| \leq \alpha h^{(1-m)/(m+1) + \beta}}\| v-u \|_{L^\infty_x}
\leq h^{-1/2} \sup_{|t| \leq \alpha h^{(1-m)/(m+1) + \beta}}\| v-u
\|_{L^2_x}.
\]
Let the energy $\EE(t) = \| v-u\|_{L^2}^2$, and compute
\begin{align*}
\EE'& =2 \Re \frac{i}{h} \int E \overline{ (v-u)} dx \\
& \leq h^{-1} h^{(1-m)/(m+1) + \beta} \| E \|_{L^2_x}^2 + h^{(m-1)/(m+1)
  + \beta} \EE,
\end{align*}
and hence by Gronwall's inequality,
\begin{align*}
\EE(t) & \leq C h^{-2m/(m+1) + \beta} \| E \|_{L^2_t L^2_x}^2 \\
& \leq C h^{(1-3m)/(m+1) + 2\beta } \| E \|_{L^\infty_{h^{(1-m)/(m+1)
      + \beta}} L^2_x }^2 \\
& \leq C h^{1 +2(K+1) \beta} \| \chi^*_\star w_0 \|_{L^2}^2 \\
& \leq C h^{2(K+1)\beta} \| \chi^*_\star w_0 \|_{L^1_x}^2.
\end{align*}
We finally conclude
\begin{align*}
\sup_{|t| \leq \alpha h^{(1-m)/(m+1) + \beta}}\| v-u \|_{L^\infty_x}
& \leq C h^{-1/2 + (K+1) \beta} \| \chi^*_\star w_0 \|_{L^1_x} \\
& \leq C |ht|^{-1/2} \| \chi^*_\star w_0 \|_{L^1_x} ,
\end{align*}
provided $| t | \leq \alpha h^{(1-m)/(m+1) + \beta}$ and $K$ is
sufficiently large that
\[
-\frac{1}{2} + (K+1) \beta \geq -\frac{1}{m+1} - \frac{\beta}{2}.
\]
The Strichartz estimates for $\Sloo(t)$ follow immediately.

\end{proof}

\subsection{The parametrix for $\Slojp^+(t)$}
The operators $\Slojp^+(t)$ are the propagator localized to outgoing 
frequencies $-ax^m \leq \xi \leq 2 \epsilon$ in the
spatial interaction region $\{ \delta h^{1/(m+1)} /2 \leq \pm x \leq
2\epsilon \}$.  We have divided the spatial interaction region into
$h$-dependent geometric regions; $\Slojp^+(t)$ is localized to 
\[
x \in h^{1/(m+1)} I_j := [h^{1/(m+1)}\delta (\omega^j -
  \omega^{j-2}) , h^{1/(m+1)}\delta (\omega^{j+1} + \omega^{j-1} )] .
\]
The symbol $\tchi((\xi + ax^m)/\gamma x^m)$ is invariant under the
rescaling operation, so after applying the rescaling operators, we are
interested in constructing a parametrix in the regions
\[
-ax^m \leq \xi \leq 2 \epsilon h^{-m/(m+1)}, \,\,  x \in I_j .
\] 

When the derivative of the effective potential $\tV'$ is large, singularities propagate
away quickly, however not uniformly so.  We introduce a loss by
constructing $\log(1/h)$ parametrices, and by eventually restricting
our construction to subcritical time scales.


We now compute how long it takes a wave packet to exit the interval
$I_j$.  Write 
\[
I_j = [y_j^-, y_j^+]:=  [\delta (\omega^j -
  \omega^{j-2}) , \delta (\omega^{j+1} + \omega^{j-1} )] , 
\]
and fix an initial point $(y, \eta)$
with $y \in I_j$, $\eta \geq -a(y_j^+)^m$.  Then recalling the
Hamiltonian system \eqref{E:Ham-q}, we have  
\[
x (t) \geq y - 2t a (y^+_{j})^m \geq \frac{1}{2} y_j^-
\]
as long as
\[
0 \leq t \leq \frac{y_j^- (y_{j}^+)^{-m}}{4a}.
\]
We have
\[
y_j^+ =  y_j^-(\omega + \O(\omega^{-1})),
\]
so that $x(t) \geq y_j^-/2$ provided
\[
0 \leq t \leq \frac{(y_j^-)^{1-m}}{4 a \omega^m} (1 + \O(\omega^{-1})).
\]
In this case, 
\[
-\partial_x \tV \geq (y_j^-/2)^{2m-1},
\]
which in turn implies
\[
\xi \geq -a (y_{j}^+)^m + t(y_j^-/2)^{2m-1} \geq b ( y_{j}^-)^m,
\]
provided
\[
t \geq 2^{2m-1}( a\omega^{m}(1 + \O(\omega^{-1})) + b) (y_j^-)^{1-m}.
\]
Choosing $a, b>0$ sufficiently small means we can assume $\eta \geq
b (y_j^-)^{m}$ after a time comparable to at most $(y_j^-)^{1-m}$.

We now compute how long it takes to leave $I_j$ assuming $y \in I_j$
and $\eta \geq b (y_j^-)^m$.  We have
\begin{align*}
x & = y + 2 t \eta + \int_0^t (t -s) B(x(s)) ds \\
& \geq y_j^- + 2 t b (y_j^-)^m + \int_0^t (t-s) B(y_j^-) ds \\
& \geq  y_j^- + 2tb (y_j^-)^m + \frac{1}{2} t^2 (y_j^-)^{2m-1} \\
& \geq y_{j}^+ 
\end{align*}
provided
\[
t \geq (y_j^-)^{1-m} \left( -2b + \sqrt{4 b^2 + 2 (y_{j}^+/y_j^- - 1)}
\right),
\]
which is again comparable to $(y_j^-)^{1-m}$.



We now estimate for $t = \alpha (y_j^-)^{1-m}$, for $\alpha>0$ to be determined:
\begin{align*}
\left| \frac{\partial x}{\partial y}(t)  \right| & \leq 1 + \int_0^t
(t-s) (4m-2) (y_{j}^+)^{2m-2} ds \left| \frac{\partial x}{\partial y}(t)
\right| \\
& \leq 1 +  (2m-1)t^2 (y_{j}^+)^{2m-2} \left| \frac{\partial x}{\partial y}(t)
\right|  \\
& \leq 1 + C_{\omega,m, a, b} \alpha^2 \left| \frac{\partial x}{\partial
    y}(t)  \right|.
\end{align*}
Choosing $\alpha>0$ sufficiently small (but independent of $h$) shows
that 
\[
\left| \frac{\partial x}{\partial y}(t)  \right|  \leq C
\]
uniformly for $t$ in this range.

With this estimate in hand, we can compute $\partial x / \partial \eta
= 2t(1 + \O(t))$ as usual, which results in the following Lemma.  In
practice, we need to gain some powers of $h$ in our parametrix
construction, so we only construct the parametrix up to time $t \sim
h^{\epsilon/2} (y_j^-)^{1-m}$ for a small $\epsilon>0$, and then iterate $C h^{-\epsilon}$
times.  After time $t \sim h^{-\epsilon/2} (y_j^-)^{1-m}$ then the
wavefront set will be outside the interval $I_j$.  Let us state the
following lemma for the short $h$-independent time scale $0 \leq t \leq \alpha
(y_j^-)^{1-m}$; we will worry about summing over the $h$-dependent
number of time intervals after
constructing the amplitude.

\begin{lemma}
\label{L:Sloojpm}
There exists $\alpha, a >0,$ and $\omega>1$ independent of $h$ and $j$
such that for each $0 \leq j \leq \O( \log(1/h))$, there is a phase function $\phi(t, x, \xi)$ satisfying
\[
\begin{cases}
\phi_t + \phi_x^2 + \tV(x;h) = 0, \\
\phi(0, x, \eta) = x \eta
\end{cases}
\]
for $x \in I_j$, $-a (y_j^+)^m \leq
\xi \leq 2 \epsilon h^{-m/(m+1)}$, and $|
t | \leq \alpha (y_j^-)^{1-m}$.

We further have 
\[
\phi_{\eta \eta} \sim 2 t(1 + \O(t)),
\]
for $| t | \leq \alpha (y_j^-)^{1-m}$.
\end{lemma}

We now construct the amplitude for the parametrix for
the operator $\Slojp^+ (t)$.  This, combined with
Lemma \ref{L:Sloojpm}, will be used to compute a
dispersion estimate, resulting in a
Strichartz estimate.  The problem is that, just as in Subsection
\ref{SS:sloo}, we are working in a marginal calculus, so to construct
the amplitude as an asymptotic series, we must restrict the range of
$t$ to depend mildly on $h$.


We again appeal to the Frobenius theorem to get a function 
$\Gamma(t,x)$ (again implicitly depending on the frequency $\xi$) satisfying
\[
\begin{cases}
\partial_t \Gamma + 2 \phi_x \partial_x \Gamma = 0, \\
\Gamma(0,x) = x.
\end{cases}
\]
We then construct $B = \sum_{j = 0}^KB_j$ for sufficiently large $K$ to
be determined (independent of $h$) with 
\[
\begin{cases}
B_0 \equiv 1, \\
B_j = -\int_0^t \phi_{xx} B_{j-1} |_{(s, \Gamma(t-s,x))} + i
B_{j-1,xx} |_{(s, \Gamma(t-s,x))}.
\end{cases}
\]
A tedious induction argument shows that $B_j$ satisfies
\[
| \partial_x^l B_j | = \O \left( \sum_{k = 1}^j \left| t^{k+j}
    x^{2km-2j-l} \right| \right).
\]



Then 
\[
w(t, x) = (2 \pi)^{-1} \int e^{i \phi(t, x, \xi) - i y \xi }
B(t, x, \xi) \chi(y, D_y)^* w_0(y) dy d \xi
\]
solves
\[
\begin{cases}
\tP w = \tE, \\
w(0, x) = \chi_\star (x, D_x)^* w_0(x),
\end{cases}
\]
where 
\[
\chi_\star = T^{-1}_h \psi( \pm  \omega^j x/h^{1/(m+1)}
 )
 ( 1 - \11_{\{\pm hD_x \geq
   1-V(x)\} } (1-\chi_{\epsilon^2} ((P-1))) ) \chi_\epsilon(x) T_h
\]
is the appropriate microlocal cutoff, and the equation is
understood to make sense for $|t| \leq \alpha (y_j^-)^{1-m}$.  
Here, the error $\tE$ is given by
\[
\tE  = (2 \pi)^{-1} \int e^{i \phi(t, x, \xi) - i y
  \xi} (-\partial_x^2 B_K -i \phi_{xx} B_K )\chi^*(y,
D_y) w_0(y) dy d \xi .
\]
That is, $\tE$ is an oscillatory integral operator with the same phase
function as $w$.  Having computed the symbol of the error term $\tE$ to be $-\p_x^2 B_K
- i \phi_{xx} B_K$, in the rescaled coordinates we have for $|t | \leq
h^{\beta/2} |x|^{1-m}$, 
\begin{align*}
 -\p_x^2 B_K
- i \phi_{xx} B_K  & = \O \left( \sum_{l = 1}^{K+1} | 
  t|^{l + K} | x|^{2ml - 2K -2} \right) \\
& = \O\left(\sum_{l = 1}^{K+1} h^{(l+K)\beta/2} | x |^{(m+1)(l-K) -2 }
\right) \\
& = \O ( h^{(1+K) \beta/2} |x |^{m-1} )
\end{align*} 
in the worst case when $l = K+1$.  Now since $| x | \leq
h^{-1/(m+1)}$, this error term is of order $\O(h^{(1 + K ) \beta /2 +
  (1-m)/(1+m) } )$, which is small as $K$ gets large.

If we now take $v = T_h w$, we see
\begin{align*}
P v & = T_h T_h^{-1} P T_h w \\
& = h^{2m/(m+1)} T_h \tP w \\
& = E,
\end{align*}
with initial conditions
\[
v(0, x) = T_h w(0,x),
\]
and where
\[
E = h^{2m/(m+1)} T_h \tE.
\]

A similar computation to Subsection \ref{SS:sloo} shows 
\begin{align*}
\sup_{0 \leq t \leq \alpha h^{ \beta/2} |y_j^-|^{1-m} }\| E
\|_{L^2} \leq C h^{2m/(m+1)} h^{\beta (1+K)/2 +(1-m)/(1+m)}\| \chi_\star^* v_0 \|_{L^2}.
\end{align*}

The following lemma contains the dispersion and Strichartz estimates
for the operators $\Slojp^+(t)$.

\begin{lemma}
\label{L:disp-slojp}

The parametrix $v(t, x)$ satisfies the dispersion estimate
\[
\| \chi_\star v \|_{L^\infty} \leq C (ht )^{-1/2}
\| \tchi v_{0} \|_{L^1},
\]
where $0< t \leq  \alpha h^{ \beta/2} |y_j^-|^{1-m} $, as well as the
corresponding Strichartz estimate
\[
\| v \|_{L^p_{ \alpha h^{ \beta/2} |y_j^-|^{1-m}  } L^q} \leq C
h^{-1/p} \|
\chi v_{0} \|_{L^2},
\]
for 
\[
\frac{2}{p} + \frac{1}{q} = \frac{1}{2}, \,\,\, q < \infty,
\]
and constants independent of $h$.

The cutoff propagator $\Slojp^+$ satisfies 
\[
\| \Slojp^+ \|_{L^2 \to L^{p}_{ \alpha h^{ \beta/2} |y_j^-|^{1-m} } L^{q}
}\leq C h^{-1/p} ,
\]
and
\[
\| \Slojp^+ \|_{L^2 \to L^{p}_{\alpha h^{(1-m)/(1+m) }} L^{q}
}\leq C h^{-(1+\beta)/p} ,
\]
for $(p,q)$ in the same range and constants independent of $h$.

\end{lemma}

The proof is exactly the same as the proof of Lemma \ref{L:disp-sloo},
with the exception of the different time interval.  If we sum over
$h^{-\beta}$ intervals of length $h^{\beta/2} | y_j^- |^{1-m}$
results in an interval of length  $h^{-\beta/2} | y_j^- |^{1-m}$.
According to Lemma \ref{L:Sloojpm} (combined with the Egorov theorem in the
$h^{-1/2+\beta}$ calculus), after this time the parametrix and the
error are both $\O(h^\infty)$.

\subsection{Proof of Proposition \ref{P:mode-str}}

In this subsection, we see how to use the computed Strichartz estimates plus the
local smoothing from \cite{ChWu-lsm} to prove Proposition \ref{P:mode-str}.

From the semiclassical Strichartz estimates, if we let $v(t, x) =
v_{lk}(th,x)$ as in Proposition \ref{P:mode-str} and rescale appropriately, we get
\[
\| \chi v_{lk} \|_{L^{2n}_T L^{2^\star} } \leq C_\beta \| \lll k \rrr^{\beta}
v_{lk}^0 \|_{L^2},
\]
for $T \leq \epsilon k^{-2/(m+1)}$, and where $\chi \in \Ci_c$ is any
smooth, compactly supported function.  Recall that according to Lemmas
\ref{L:L}, we already have perfect Strichartz
estimates for $(1-\chi) 
v_{lk}$ if $\chi \equiv 1$ near $x = 0$.  Further, by Lemma
\ref{L:Shi}, we have perfect Strichartz estimates for large
frequencies and small $x$: if $\psi(\xi) \equiv 1$ near $0$, $\chi
(1-\psi(-h^2 \Delta)) v_{lk}$ obeys perfect Strichartz estimates.

Let $\chi$ and $\psi$ be such cutoffs.  In order to estimate $\chi \psi v_{lk}$, we
employ a duality trick (see \cite{BGH}) together with the local smoothing estimates
from \cite{ChWu-lsm}.  Let $\phi(s) \in \Ci_c$ be a compactly
supported function such that
\[
\sum_{j = 0}^{k^{2/(m+1)}} \phi ( k^{2/(m+1)} t - j ) \equiv 1, \,\,\,
0 \leq t \leq \epsilon.
\]
Set $U_j = \phi ( k^{2/(m+1)} t - j ) \chi \psi v_{jk}$.  We have
\[
(D_t + P_k) U_j = W_j' + W_j'',
\]
where
\[
W_j' = i k^{2/(m+1)} \phi' ( k^{2/(m+1)} t - j ) \chi \psi v_{jk},
\]
and
\[
W_j'' = \phi ( \chi''  + 2 \chi' \p_x) \psi v_{lk}.
\]
The important thing to observe is that $W_j''$ is supported away from
$x = 0$, so the standard $1/2$ derivative local smoothing estimates
hold (see Theorem \ref{T:lsm}).  Let $\chi_1 \in \Ci_c$ satisfy
$\chi_1 \equiv 1$ on $\supp \chi$, and $\chi_2 \in \Ci_c$ satisfy
$\chi_2 \equiv 1$ on $\supp \chi'$, $\supp \chi_2$ away from $x = 0$.
We have $\chi_1 U_j = U_j$, $\chi_1 W_j' = W_j'$, and $\chi_2 W_j'' =
W_j''$.  Using the Duhamel formula, set
\[
U_j' = \chi_1 \int_{(j-1) \epsilon k^{-2/(m+1)} }^t e^{-i(t-s) P_k }
\chi_1 W_j'(s) ds,
\]
and
\[
U_j'' = \chi_1 \int_{(j-1) \epsilon k^{-2/(m+1)} }^t e^{-i(t-s) P_k }
\chi_2 W_j''(s) ds,
\]
so that $U_j' + U_j'' = U_j$.  By the Christ-Kiselev lemma \cite{CK},
it suffices to consider 
\[
\overline{U}_j' = \chi_1 \int_{(j-1) \epsilon k^{-2/(m+1)} }^{(j+1) \epsilon k^{-2/(m+1)} } e^{-i(t-s) P_k }
\chi_1 W_j'(s) ds,
\]
and similarly for $W_j''$.  Let $I = [(j-1) \epsilon k^{-2/(m+1)} ,
(j+1) \epsilon k^{-2/(m+1)} ]$ be the time interval in the integral
above.  We apply the Strichartz estimates to get
\[
\| \overline{U}_j' \|_{L^{2n}_I L^{2^\star}} \leq C k^{\beta} \left\| \int_{(j-1) \epsilon k^{-2/(m+1)} }^{(j+1) \epsilon k^{-2/(m+1)} } e^{is P_k }
\chi_1 W_j'(s) ds \right\|_{L^2}, 
\]
and similarly for $W_j''$.  The dual estimates to Theorem \ref{T:lsm}
then 
yield
\[
\| \overline{U}_j' \|_{L^{2n}_I L^{2^\star}} \leq C k^{\beta-1/(m+1)} \| 
W_j'  \|_{L^2 L^2}, 
\]
and (again because $\chi_2$ is supported away from $x = 0$)
\[
\| \overline{U}_j'' \|_{L^{2n}_I L^{2^\star}} \leq C k^{\beta-1/2} \| 
W_j''  \|_{L^2 L^2}.
\]
By the Christ-Kiselev lemma \cite{CK}, the same estimates hold for
$U_j'$ and $U_j''$.  Squaring and summing in $j$, using $\ell^2
\subset \ell^{2n}$, yields
\begin{align*}
\| v_{lk} \|_{L^{2n}_{\epsilon} L^{2^\star}}^2 & \leq C \sum_{j =
  0}^{k^{2/(m+1)}} ( \| U_j' \|_{L^{2n}_{\epsilon} L^{2^\star}}^2 + \|
U_j'' \|_{L^{2n}_{\epsilon} L^{2^\star}}^2) \\
& \leq C \sum_{j =
  0}^{k^{2/(m+1)}} (k^{2 \beta - 2/(m+1)} \| 
W_j'  \|_{L^2 L^2}^2 +  k^{2\beta-1} \| 
W_j''  \|_{L^2 L^2}^2 ) \\
& \leq C ( k^{2 \beta + 2/(m+1)}  \| \chi v_{jk} \|_{L^2_{\epsilon}
  L^2}^2 +  k^{2\beta-1} \| \chi_2 \lll D_x \rrr v_{lk}
\|_{L^2_\epsilon L^2}^2 ) \\
& \leq C k^{2\beta} \| v_{lk}^0 \|_{L^2}^2.
\end{align*}
This proves Proposition \ref{P:mode-str}.


\section{Quasimodes}

In this section we construct quasimodes for the model operator near
$(0,0)$ in the transversal phase space, and then use these quasimodes
to show the Strichartz estimates are near-sharp, in the sense
described in Corollary \ref{C:C1a}.

Consider the model operator
\[
P = -h^2 \partial_x^2 - m^{-1} x^{2m}
\]
locally near $x = 0$.  We will construct quasimodes which are
localized very close to $x = 0$, so this should be a decent
approximation.  It is well-known (see \cite{ChWu-lsm}) that the operator
\[
\tilde{Q} = -\partial_x^2 + x^{2m}
\]
has a unique ground state $\tilde{Q} v_0 = \lambda_0 v_0$, with
$\lambda_0>0$, and $v_0$ is a Schwartz class function.  Then, by rescaling, we find the function $v(x) = v_0 ( x
h^{-1/(m+1)})$ is an un-normalized eigenfunction for the equation
\[
(-h^2 \partial_x^2 + x^{2m} ) v = h^{2m/(m+1)} \lambda_0 v.
\]
Complex scaling then suggests there are resonances with imaginary part $c_0
h^{2m/(m+1)}$.  We use a complex WKB approximation to get an explicit
formula for a localized approximate resonant state, however, as we
shall see, it is not a very good approximation.  Nevertheless, since
we will eventually be averaging in time, it is sufficient for our applications.

Let $E_0 = (\alpha + i \mu)h^{2m/(m+1)} $, $\alpha, \mu>0$
independent of $h$.  Let the phase function 
\[
\phi(x) = \int_0^x (E + m^{-1} y^{2m})^{1/2} dy,
\]
where the branch of the square root is chosen to have positive
imaginary part.  Let 
\[
u(x) = (\phi')^{-1/2} e^{i \phi  / h},
\]
so that
\[
(hD)^2 u = (\phi')^2 u + f u,
\]
where
\begin{align*}
f & = (\phi')^{1/2} (hD)^2 (\phi')^{-1/2} \\
& = -h^2 \left( \frac{3}{4} (\phi')^{-2} (\phi'')^2 - \frac{1}{2}
  (\phi')^{-1} \phi ''' \right).
\end{align*}

\begin{lemma}
The phase function $\phi$ satisfies the following properties:
\begin{description}

\item[(i)]  There exists $C>0$ independent of $h$ such that 
\[
| \Im \phi | \leq C\begin{cases} h(1 + \log(x/h^{1/2} )), \quad m =
  1, \\
 h, \quad m \geq 2.
\end{cases}
\]
In particular, if $| x | \leq C h^{1/(m+1)}$, $| \Im \phi| \leq C'$
for some $C'>0$ independent of $h$.

\item[(ii)]  There exists $C>0$ independent of $h$ such that 
\[
C^{-1} \sqrt{ h^{2m/(m+1)} + x^{2m} } \leq | \phi'(x) | \leq C \sqrt{
  h^{2m/(m+1)} + x^{2m} }
\]

\item[(iii)]
\[
\begin{cases}
\phi' = (E + m^{-1} x^{2m})^{1/2}, \\
\phi'' =  x^{2m-1} (\phi')^{-1}, \\
\phi''' = \left( (1 - m^{-1} ) x^{4m-2} + E  (2m-1) x^{2m-2}
\right) ( \phi')^{-3}, 
\end{cases}
\]
In particular,
\[
f = -h^2 x^{2m-2} \left( \left( \frac{1}{4}  + \frac{1}{2 m}
  \right) x^{2m} - \left( m - \frac{1}{2} \right) E \right) (\phi'
)^{-4}.
\]

\end{description}

\end{lemma}

\begin{proof}
For (i) we write $\phi' = s + it$ for $s$ and $t$ real valued, and then
\[
E + m^{-1} x^{2m} = s^2 - t^2 + 2 i st.
\]
Hence
\[
s^2 \geq s^2 - t^2 = \alpha h^{2m/(m+1)} + m^{-1}x^{2m} ,
\]
so that
\[
t = \frac{\mu h^{2m/(m+1)}}{2s} \leq \frac{\mu
  h^{2m/(m+1)}}{2\sqrt{h^{2m/(m+1)} \alpha + m^{-1} x^{2m}}}.
\]
Then
\begin{align*}
| \Im \phi (x) | & \leq \int_0^{|x|} \phi'(y) dy \\
& \leq C \int_0^{h^{1/(m+1)}} h^{m/(m+1)} dy + C \int_{h^{1/(m+1)}}^x
h^{2m/(m+1)} y^{-m} dy \\
& = \begin{cases} \O ( h(1 + \log (x/h^{1/2}))), \quad m = 1, \\
\O(h), \quad m >1. \end{cases}
\end{align*}

Parts (ii) and (iii) are simple computations.

\end{proof}

In light of this lemma, $| u (x) |$ is comparable to $| \phi'
|^{-1/2}$, provided $| x | \leq C h^{1/2}$ when $m=1$.  We are only
interested in sharply localized quasimodes and in the case $m \geq 2$, so
let $\gamma = h^{1/(m+1)}$, choose $\chi(s) \in \Ci_c( \reals)$ such
that $\chi \equiv 1$ for $| s | \leq 1$ and $\supp \chi \subset
[-2,2]$.  Let
\[
\tu(x) = \chi(x/\gamma) u(x),
\]
and compute for $q \geq 2$: 
\begin{align*}
\| \tu \|_{L^q}^q & = \int_{| x | \leq 2 \gamma} \chi(x/\gamma)^q | u
|^q dx \\
& \sim \int_{|x| \leq 2\gamma} \chi(x/\gamma)^q | \phi' |^{-q/2} dx \\
& \sim h^{1/(m+1)} h^{ -qm/2(m+1)} \\
& \sim h^{(2-qm)/2(1+m)}.
\end{align*}
In particular,
\[
\| \tu \|_{L^2} \sim h^{(1-m)/2(1+m)},
\]
and so
\[
\| \tu \|_{L^q} \sim h^{(2/q -1)/2(m+1)} \| \tu \|_{L^2}.
\]

Further, $\tu$ satisfies the following equation:
\begin{align*}
(hD)^2 \tu & = \chi(x/\gamma) (hD)^2 u + [(hD)^2, \chi(x/\gamma)] u \\
& = (\phi')^2 \tu + f \tu + [(hD)^2, \chi(x/\gamma)] u \\
& = (\phi')^2 \tu + R,
\end{align*}
where 
\[
R = f \tu + [(hD)^2, \chi(x/\gamma)] u.
\]

\begin{lemma}
The remainder $R$ satisfies
\begin{equation}
\label{E:R-remainder}
\| R \|_{L^2} = \O (h^{2m/(m+1)}) \| \tu \|_{L^2}.
\end{equation}

\end{lemma}

\begin{proof}
We have already computed the function $f$, which is readily seen to
satisfy 
\[
\| f \|_{L^\infty(\supp (\tu ))} = \O(h^{2m/(m+1)}),
\]
since $\supp (\tu) \subset \{ | x | \leq 2 h^{1/(m+1)} \}$.

On the other hand, since $\| \tu \|_{L^2} \sim h^{(1-m)/2(1+m)}$, we need only show that
\[
\| [(hD)^2, \chi(x/\gamma)] u\|_{L^2} \leq C h^{(3m+1)/2(m+1)}.
\]
We compute:
\begin{align*}
[(hD)^2, \chi(x/\gamma)] u & = -h^2 \gamma^{-2} \chi'' u + 2\frac{h}{i}
\gamma^{-1} \chi' hD u \\
& = -h^2 \gamma^{-2} \chi'' u + 2\frac{h}{i}
\gamma^{-1} \chi'  \left(-\frac{h}{2i} \frac{\phi''}{\phi'} + \phi'
\right) u \\
& = -h^2 \gamma^{-2} \chi'' u + 2\frac{h}{i}
\gamma^{-1} \chi' \left( -\frac{h}{2i} \frac{ x^{2m-1}}{(\phi')^2} +
  \phi' \right) u.
\end{align*}
The first term is estimated:
\[
\| h^2 \gamma^{-2} \chi'' u \|_{L^2} = \O(h^{2m/(m+1)}) \| u
\|_{L^2(\supp ( \tu))} =\O(h^{(3m+1)/2(m+1)}).
\]
Similarly, the remaining two terms are estimated:
\begin{align*}
\Bigg\| 2\frac{h}{i} &
\gamma^{-1} \chi' \left( -\frac{h}{2i} \frac{ x^{2m-1}}{(\phi')^2} +
  \phi' \right) u \Bigg\|_{L^2} \\
& = \O(h^{m/(m+1)} h^1 h^{(2m-1)/(m+1)}
h^{-2m/(m+1)}) \| u
\|_{L^2(\supp ( \tu))}  \\
& \quad + \O(h^{m/(m+1)} h^{2m/(m+1)} )  \| u
\|_{L^2(\supp ( \tu))}  \\
& = \O(h^{(3m+1)/2(m+1)}).
\end{align*}

\end{proof}

\subsection{Saturation of Strichartz estimates}
\label{SS:saturation}

In this subsection, we study Strichartz estimates for the separated
Schr\"odinger equation given the specific choice of initial conditions
in the form of quasimodes.

Now it is well known that for any $k$, there exists a spherical
harmonic $v_k$ of order $k$ which saturates
Sogge's bounds (Theorem \ref{T:Sogge}):
\[
-\Delta_{\SS^d} v_k = (k)(k+d -1) v_k, \quad \| v_k \|_{L^{2(d+1)/(d-1)}} \sim k^{(d-1)/2(d+1)} \|
v_k \|_{L^2}.
\]

Let $\lambda_k = k (k + n -2)$, $k \gg 1$, $h = \lambda_k^{-1/2}$, let
$\tu$ be the associated transversal quasimode constructed in the
previous section, 
and let 
\[
\phi_0(x, \theta) = v_k(\theta) \tu(x).
\]

Let $\phi(t, x, \theta) = e^{it \tau} \phi_0$ for some $\tau \in \cx$
to be determined.  Since the support of $\tu$ is very small, contained
in $\{ | x | \leq h^{1/(m+1)} / \kappa \}$, we have 
\[
A^{-2} = (1 + x^{2m} )^{-1/m} = 1 -\frac{1}{m} x^{2m} +
\O(h^{4m/(m+1)})
\]
on $\supp \tu$.  Then
\begin{align*}
(D_t + \tDelta) \phi & = P_k \phi \\
& = ( \tau - D_x^2 - A^{-2} \lambda_k - V_1(x) ) \phi \\
& = \lambda_k e^{it \tau} e^{i k \theta} \left[\left( \tau
    \lambda_l^{-1} - (\lambda^{-1}_k D_x^2 + 1 - \frac{1}{m}
  x^{2m} ) \right) \tu  +  \O( k^{-2}) \tu \right] \\
& = \lambda_k e^{it \tau} e^{i k \theta}  \left[ \left( \tau \lambda_k^{-1} - 1 -
    E_0 \right) \tu + R + \O( k^{-2}) \tu \right],
\end{align*}
where $R$ satisfies the remainder estimate \eqref{E:R-remainder}.  Set
\[
\tau = \lambda_k (1+ E_0) = \lambda_k (1 + \alpha k^{-2m/(m+1)}) + i \mu
k^{2/(m+1)}(1+\O(k^{-1})), \,\,\, \alpha, \mu >0
\]
so that we have
\[
\begin{cases}
(D_t + \tDelta) \phi = \tR, \\
\phi(0, x, \theta) = \phi_0
\end{cases}
\]
with
\begin{equation}
\label{E:tR}
\tR = \lambda_k e^{i t \tau } v_k (R(x, k) + \O(k^{-2}) \tu ).
\end{equation}

We compute the endpoint Strichartz estimate on an arbitrary time
interval $[0,T]$, with $p=2$, $q = 2^\star = 2n/(n-2)$ for $n \geq 3$:
\begin{align}
\| \phi \|_{L^2([0,T])L^q}^2 & = \int_0^T \| e^{it \tau} \phi_0
\|^2_{L^q} dt \notag \\
& = \int_0^T e^{-2t\Im \tau} \|  \phi_0 \|_{L^q}^2 dt \notag
\\
& = \frac{1 - e^{-2 T\Im \tau}}{2 \Im \tau} \| \phi_0
\|_{L^q}^2 \notag \\
& = \frac{1 - e^{-2 T\Im \tau}}{ 2 \Im \tau } \| 
\phi_0 \|_{L^q}^2 \notag \\
& \sim \frac{1 - e^{-2 T\Im \tau}}{ 2 \Im \tau } k^{(1-2/q)/(m+1)}
k^{(n-2)/n} \| \tu
\|_{L^2(\reals)}^2 \| v_k \|_{L^2(\SS^{n-1})}^2 \notag \\
& \sim k^{2\eta(m,n)} \| \phi_0 \|_{L^2}^2 \notag \\
& \sim \| (-\Delta_{\SS^{n-1}})^{\eta(m,n)} \phi_0 \|_{L^2}^2, \label{E:sharp-str}
\end{align}
where
\[
\eta(m,n) = \frac{1}{2(m+1)} \left( m\left( 1 - \frac{2}{n} \right) -1
\right).
\]

Now let $L(t)$ be the unitary Schr\"odinger propagator:
\[
\begin{cases}
(D_t + \tDelta) L = 0, \\
L(0) = \id,
\end{cases}
\]
and write using Duhamel's formula:
\[
\phi(t) = L(t) \phi_0 + i \int_0^t L(t) L^*(s) \tR(s) ds =: \phi_{\text{h}} + \phi_{\text{ih}},
\]
where $\phi_{\text{h}}$ and $\phi_{\text{ih}}$ are the homogeneous and
inhomogeneous parts respectively.  We want a lower bound on the
homogeneous Strichartz estimates, for which we need an upper bound on the
inhomogeneous Strichartz estimates.

Let us now assume for the purposes of contradiction that a better Strichartz estimate than that in
Corollary \ref{C:C1a} holds for all $\beta >0$.  That is, we
assume for each $\beta>0$, there exists $C_\beta$ such that 
\[
\| L(t) u_0 \|_{L^2 ([0,T]) L^{2^\star}} \leq C_\beta \| \lll
-\Delta_{\SS^{n-1}} \rrr^{r+ \beta} u_0 \|_{L^2},
\]
for some $r < \eta(m,n)/2$.  

In dimension $n = 2$, we take as usual $p >2$, $2 \leq q < \infty$,
and we immediately arrive at a contradiction to the scale-invariant case.

For dimension $n \geq 3$, we take $\beta>0$ sufficiently small that
$r + \beta < \eta(m,n)/2$, and we then have the complementary inhomogeneous Strichartz estimate: if
$v$ solves
\[
\begin{cases}
(D_t + \tDelta) v = F, \\
v(0) = 0,
\end{cases}
\]
then
\[
\| v \|_{L^2 ([0,T]) L^{2^\star}} \leq C \| \lll
-\Delta_{\SS^{n-1}} \rrr^{r + \beta}  F \|_{L^1([0,T])
  L^2 }.
\]

For the inhomogeneous part corresponding to our quasimode initial
data, we have $F = \tR$, with $\tR$ computed in \eqref{E:tR}.  Then 
\begin{align*}
\| & \phi_{\text{ih}} \|_{L^2([0,T]) L^{2^\star}} \\
& \leq C T^{1/2} \| \tR
\|_{L^2([0,T]) L^2} \\
& \leq C k^2 T^{1/2} \left( \int_0^T e^{-2 t \Im \tau } \|\lll
-\Delta_{\SS^{n-1}} \rrr^{r + \beta}  v_k (R(x, k) +
  \O(k^{-2}) \tu ) \|_{L^2}^2 dt  \right)^{1/2} \\
& \leq C k^2 k^{-2m/(m+1)} \left( \frac{1-e^{-2T \Im \tau}}{2\Im \tau}
\right)^{1/2} \| \lll
-\Delta_{\SS^{n-1}} \rrr^{r + \beta} \phi_0 \|_{L^2}.
\end{align*}
Recalling that $\Im \tau \sim k^{-2/(m+1)}$, if $T = \epsilon^2
k^{-2/(m+1)}$, we have
\begin{equation}
\label{E:ih-sharp-str}
\| \phi_{\text{h}} \|_{L^2([0,T]) L^{2^\star}} \leq C \epsilon \| \lll
-\Delta_{\SS^{n-1}} \rrr^{r + \beta} \phi_0 \|_{L^2} .
\end{equation}
Now, if $\epsilon >0$ is sufficiently small, but independent of $k$,
we have
\[
1 \geq 1-e^{-2T \Im \tau} \geq c_0,
\]
for some $c_0>0$, so that for this choice of $T$, we still have the
estimate \eqref{E:sharp-str}.  Combining \eqref{E:sharp-str} with
\eqref{E:ih-sharp-str} we have
\begin{align*}
C \| \lll
-\Delta_{\SS^{n-1}} \rrr^{r + \beta} \phi_0 \|_{L^2} & \geq \| L(t)
\phi_0 \|_{L^2([0,T]) L^{2^\star}} \\
& \geq \| \phi(t)
\|_{L^2([0,T]) L^{2^\star}} - \|   \phi_{\text{ih}}\|_{L^2([0,T])
  L^{2^\star}}  \\
& \geq C^{-1} \| \lll
-\Delta_{\SS^{n-1}} \rrr^{\eta(m,n)/2} \phi_0 \|_{L^2} ,
\end{align*}
for some constant $C>0$ independent of $k$.  But this is a
contradiction, since $r + \beta < \eta(m,n)/2$.  This proves the
near-sharpness of Corollary \ref{C:C1a}.

\bibliographystyle{alpha}
\bibliography{deg-str-bib}

\end{document}